%% file: Unfitted_IGA_arxiv.tex
\documentclass[11pt,a4paper,reqno]{amsart}
\usepackage[margin=1.1in]{geometry}
\usepackage{amssymb,amsmath,graphicx,amsfonts}
\usepackage{mathrsfs}
\usepackage{multirow}
\usepackage{epsfig}
\usepackage{color}
\usepackage{enumerate,enumitem}
\usepackage{empheq}
\usepackage{cases}
\usepackage{cite}
\usepackage{amsgen,amscd}
\usepackage{booktabs}
\usepackage{latexsym,euscript}
\usepackage{bm,hyperref,subcaption}
\usepackage{stmaryrd}
\usepackage{placeins}
\usepackage{algorithm}
\usepackage{algpseudocode} % from algorithmicx
\allowdisplaybreaks
\newtheorem{theorem}{Theorem}[section]
\newtheorem{assumption}[theorem]{Assumption}
\newtheorem{lemma}[theorem]{Lemma}
\newtheorem{proposition}[theorem]{Proposition}
\newtheorem{remark}[theorem]{Remark}
\theoremstyle{definition}

\newtheorem{corollary}{Corollary}

\numberwithin{equation}{section}

\input{macros.tex}

\newcommand{\dist}{\operatorname{dist}}

\newcommand{\CT}{\mathcal{T}}          
\newcommand{\CI}{\mathcal{I}}

\newcommand{\HCT}{\hat{\mathcal{T}}}
\newcommand{\TCI}{\widetilde{\mathcal{I}}}
\newcommand{\Tu}{\tilde{u}}

\def\TOmega{\widetilde{\Omega}_h}

\newcommand{\RNum}[1]{\uppercase\expandafter{\romannumeral #1\relax}}

\begin{document}

\title{U\MakeLowercase{n}C\MakeLowercase{ut} FEM with PHT-splines on unfitted T-meshes} 

\thanks{This work is partially supported by the National Natural Science Foundation of China (project no. 12525111) and the CAS AMSS-PolyU Joint Laboratory of Applied Mathematics.}

\author[]{Falai Chen}
\address{Falai Chen: School of Mathematics, University of Science and Technology of China.}
\email{chenfl@ustc.edu.cn}

\author[]{Nuo Lei}
\address{Nuo Lei: Department of Applied Mathematics, The Hong Kong Polytechnic University, Hung Hom, Hong Kong. Hua Loo-Keng Center for Mathematical Sciences, Academy of Mathematics and Systems Science, Chinese Academy of Sciences, Beijing, China.}
\email{nuo97.lei@polyu.edu.hk}

\author[]{Buyang Li}
\address{Buyang Li: Department of Applied Mathematics, The Hong Kong Polytechnic University, Hong Kong.}
\email{buyang.li@polyu.edu.hk}

\keywords{Unfitted finite element method, unfitted mesh, error estimates, adaptive PHT-splines}
\subjclass[2020]{65D07, 65N15, 65N50, 65N85}

\maketitle

\begin{abstract}
We propose a novel unfitted finite element method based on polynomial splines over hierarchical T-meshes (PHT-splines) for elliptic partial differential equations (PDEs) posed on domains whose boundaries are described implicitly by a level-set function. Exploiting the enhanced smoothness of the spline space, we formulate the method by minimizing the $L^2$ residual of the PDE over an extended neighborhood of the physical domain. This formulation allows the method to be implemented entirely on the background mesh using standard element-wise quadrature. It yields a symmetric positive-definite linear system and avoids the geometric construction of curved cut subregions in boundary elements, as well as the associated nontrivial numerical integration. We establish well-posedness of the numerical scheme and derive optimal \emph{a priori} error estimates in the $H^1$ norm. In addition, we employ a residual-based \emph{a posteriori} error indicator to guide local adaptive refinement of the hierarchical T-meshes. Numerical experiments corroborate the theoretical results and demonstrate the accuracy and efficiency of the proposed approach.
\end{abstract}

\setlength\abovedisplayskip{4pt}
\setlength\belowdisplayskip{4pt}

\section{Introduction}
\label{sec:intro}
The numerical solution of partial differential equations (PDEs) defined on complex or evolving domains $\Omega \subset \mathbb{R}^d$ constitutes a cornerstone of contemporary computational science. Traditionally, finite element analysis relies on \emph{body-fitted} discretisations, where the triangulation $\mathcal{T}_h$ is constrained to conform to the physical boundary $\partial\Omega$. While robust for static and simple geometries, this paradigm poses significant computational challenges in realistic scenarios. The construction of high-quality, shape-regular meshes that accommodate high-order approximations is often a prohibitive task, frequently dominating the total computational budget of the simulation pipeline.

This bottleneck is particularly pronounced in multi-query environments, such as shape optimisation, inverse problems, and uncertainty quantification, where the governing PDE must be evaluated over an extensive ensemble of geometric configurations. In these contexts, the overhead of repeated mesh generation and the subsequent verification of mesh validity can eclipse the cost of the linear solvers themselves, impeding the development of fully automated workflows. Furthermore, for time-dependent problems involving moving interfaces or evolving boundaries, classical Arbitrary Lagrangian--Eulerian (ALE) frameworks often suffer from progressive mesh distortion. Maintaining the integrity of the numerical scheme necessitates frequent remeshing and sophisticated field interpolation (projection) between successive grids, processes that are not only computationally intensive but also susceptible to introducing numerical dissipation and compromising stability.

To circumvent these geometric constraints, \emph{unfitted} finite element methods have emerged as a powerful alternative by effectively decoupling the functional discretisation from the underlying domain geometry. By employing a fixed background mesh $\mathcal{T}_h$ and representing the boundary $\partial\Omega$ implicitly (e.g., via a level-set function), these methods eliminate the need for body-conforming mesh updates. This framework provides a versatile foundation for handling topological transitions and large deformations while retaining the convergence properties of classical finite element spaces.

Among the disparate unfitted formulations, the Immersed Finite Element Method (IFEM) \cite{Guo2019,Guo2020}, the Extended Finite Element Method (XFEM) \cite{Moes1999,Xie2021XFEM}, and the Cut Finite Element Method (CutFEM) \cite{Burman2015,Burman2025} have gained prominence. 
Furthermore, the unfitted methods have also been applied to the advection-diffusion equation \cite{Ma2022}, elliptic equations on hypersurfaces \cite{Olshanskii2016}, elliptic interface problems \cite{Wu2011,Qiu2016,Hu2026}, fluid-structure interaction problems \cite{Burman2014}, 
shape optimization \cite{Gong2026} and time-harmonic wave equations \cite{Chen2024Maxwell,LiWu2025}.
A fundamental challenge inherent in these approaches is the \emph{small cut-cell problem}: when the boundary intersects the background elements arbitrarily, the resulting system can become severely ill-conditioned as the measure of the intersection approaches zero. 

Two primary paradigms have been established to restore numerical robustness. The first is the \emph{cell-merging} strategy, exemplified by the work of Chen et al. \cite{Chen2021,Chen2023,Chen2024}, which algebraically or geometrically aggregates small boundary elements with adjacent interior macro-elements. This approach has been shown to effectively mitigate the $p$-dependence of the condition number in $hp$-FEM settings, ensuring geometric stability. The second paradigm is the \emph{variational stabilisation} framework, notably the ghost penalty method developed by Burman and co-authors \cite{Burman2010,Burman2025_review}. By augmenting the bilinear form with penalisation terms on the face facets of cut elements, this approach restores uniform coercivity and ensures that the condition number remains bounded independently of the interface's relative position to the mesh. Similar minimal stabilisation techniques have also been extended to Unfitted Isogeometric Analysis (IGA) \cite{Buffa2020,Buffa2023} and least-squares Nitsche formulations \cite{Elfverson2019} to preserve optimal convergence rates on trimmed geometries.

Although these methods significantly reduce the need for mesh generation, they also introduce new implementation challenges. In particular, they often require numerical integration over geometrically intricate subregions induced by the intersection between the domain boundary and background elements. 
This typically necessitates geometric intersection algorithms and specialised quadrature rules, which are computationally expensive and difficult to implement robustly, especially in three dimensions. Moreover, such geometry-dependent quadrature procedures may introduce additional errors that are often not explicitly accounted for in existing error analyses of unfitted finite element methods.

To avoid these difficulties in both implementation and analysis, Duprez and Lozinski proposed a novel fictitious domain approach, referred to as $\phi$-FEM \cite{Duprez2020,phifem_heat,phifem_stokes,phifem_poster}, where $\phi$ denotes the level-set function defining the domain. This formulation belongs to the class of UnCut finite element methods, in which all integrals are evaluated over the full background elements using standard quadrature rules, thereby eliminating the need for geometry-dependent integration procedures. Stability and error estimates for $\phi$-FEM have been established for the Poisson equation under the assumption that the stabilisation parameters are chosen sufficiently large, depending on constants arising from trace inequalities, Poincar\'e inequalities, and inverse estimates. In practice, however, determining appropriate values for these parameters remains challenging for general elliptic problems.

To address the issue of parameter selection, Hu, Li, and Yang \cite{HuLiYang} proposed a robust, parameter-free least-squares UnCut finite element method that guarantees stability and optimal-order convergence without requiring large stabilisation parameters. However, in order to accommodate standard $C^0$ finite element spaces, their approach is formulated through the minimisation of a discrete $H^{-1}$-residual of the governing PDE. While effective, this naturally necessitates a dual (mixed) formulation involving an auxiliary variable, thereby increasing both the algebraic dimension and the computational complexity of the resulting linear system.

Motivated by the need for a more computationally efficient framework, this paper introduces a novel UnCut formulation that leverages the high-order regularity of $C^1$-continuous PHT-splines \cite{PHT-spline,Wang2011}. By exploiting $C^1$ continuity, we can directly discretise the PDE in its strong form, effectively preserving the parameter-free advantages of the UnCut framework while entirely circumventing the need for dual variables or mixed systems. To illustrate this approach,  we consider the Poisson problem on a domain $\overline\Omega \subset (0,1)^d$ defined by a $C^{k+1}$ level-set function $\phi:[0,1]^d \to \mathbb{R}$, with $d\in\{2,3\}$, i.e., 
\begin{equation}\label{poisson}
\left\{
\begin{aligned}
-\Delta u &= f \quad &&\text{in } 
\Omega = \{\mathbf x \in (0,1)^d:  \phi(\mathbf x) < 0 \}, \\
 u &= 0 \quad &&\text{on } 
\partial \Omega = \{\mathbf x \in (0,1)^d: \phi(\mathbf x) = 0 \}.  
\end{aligned}
\right.
\end{equation}

Specifically, we consider a background hierarchical T-mesh $\mathcal{T}_h$ defined over a computational domain $\Omega_h$ such that 
\[
\Omega \subset \Omega_h \subset [0,1]^d, 
\qquad 
\operatorname{dist}(\partial\Omega_h,\partial\Omega)=O(h),
\]
where $h$ denotes the mesh size. 
The approximation space is then constructed as
\[
S_h = \{ w_h = \phi_h v_h : v_h \in V_h \},
\]
where $V_h$ is the discrete spline space defined in \eqref{eq:Vh}, and $\phi_h \in V_h$ is a $k$-th degree PHT-spline approximation of $\phi$ such that 
\[
\| \phi - \phi_h \|_{L^2(\Omega_h)} \le C h^{k+1}, 
\qquad 
\| \phi - \phi_h \|_{H^1(\Omega_h)} \le C h^k,
\]
for a constant $C > 0$ (e.g., $\phi_h$ may be chosen as the global spline interpolant or the $L^2$-projection of $\phi$). 
By construction, every function in $S_h$ vanishes on the discrete interface $\Gamma_h=\{\phi_h=0\}$. Thus, the homogeneous Dirichlet condition is imposed exactly on $\Gamma_h$ and only approximately on the physical boundary $\Gamma=\partial\Omega$.

We define the numerical solution $u_h$ to problem \eqref{poisson} as the minimizer of the following functional:
\begin{align}\label{minimization}
u_h=\arg\min_{w_h\in S_h} \Big(
\| \Delta w_h + f \|_{L^2(\Omega_h)}^2 
+j_h(w_h, w_h) \Big),
\end{align}
where $j_h(\cdot, \cdot)$ is a jump-type stabilization term defined in \eqref{penalty}. 
For simplicity, $f$ is understood as a fixed bounded extension to
$\widehat\Omega$, still denoted by $f$.
This quadratic minimization problem is equivalent to the following weak formulation: 
find $u_h\in S_h$ such that 
\begin{equation}\label{scheme}
(\Delta u_h, \Delta v_h)_{L^2(\Omega_h)} + j_h(u_h,v_h)
= -(f,\Delta v_h)_{L^2(\Omega_h)} 
\quad\forall\, v_h\in S_h .
\end{equation}
This formulation is used for practical implementation. 
However, the optimal-order $H^1$ convergence of this scheme---based on minimizing an $L^2$-residual over a $C^1$ finite element space---is not immediate and requires careful analysis.

The main contributions of this paper are threefold:
\begin{itemize}
\item \textit{Formulation:} We introduce a novel unfitted finite element method in which no numerical integration over curved boundaries or cut subregions of boundary elements is required. All computations are performed on the \emph{entire} background mesh using standard element-wise quadrature. This is achieved without introducing dual formulations or auxiliary unknowns; instead, we exploit $C^1$ trial spaces constructed via PHT-splines, leading to a method that is both simple to implement and computationally efficient.\vspace{5pt}

\item \textit{Theory:} We establish a rigorous error analysis for the proposed method. In particular, we prove existence and uniqueness of the discrete solution and derive optimal-order error estimates in the $H^1$ norm. This is nontrivial since the formulation is based on minimizing the $L^2$ residual rather than directly enforcing an $H^1$-coercive variational structure.\vspace{5pt}

\item \textit{Adaptivity:} PHT-splines are defined on hierarchical T-meshes that admit T-junctions, enabling efficient local \(h\)-refinement via a simple \emph{cross-insertion} strategy. Leveraging this flexibility, we introduce a residual-based \emph{a posteriori} error indicator for each cell \(T\in \mathcal{T}_h\),
\[
\eta_T := \|\Delta u_h + f\|_{L^2(T)},
\]
and demonstrate through numerical experiments that it provides an effective refinement criterion on hierarchical T-meshes. 
\end{itemize}

The remainder of this paper is organised as follows. Section~\ref{sec:mesh_scheme} introduces the hierarchical T-spline spaces and presents the proposed least-squares UnCut FEM scheme. Section~\ref{sec:main} establishes stability and derives \emph{a priori} error estimates. In Section~\ref{sec:poster}, we construct a residual-based \emph{a posteriori} error indicator and prove its reliability and efficiency. Section~\ref{sec:test} reports numerical experiments that confirm the optimal convergence rates and illustrate the performance of the adaptive mesh refinement strategy. Finally, concluding remarks are given in Section~\ref{sec:conclusion}.

\section{UnCut FEM for the Poisson equation}
\label{sec:mesh_scheme} 
% --- Part 1: Definition of the Hierarchical T-mesh and Spline Space ---
We consider a hierarchical T-mesh {$\hat{\mathcal{T}}_h$ } as the computational background mesh over the rectangular domain $\hat{\Omega} = [0,1]^d$. 
A T-mesh facilitates local refinement by permitting T-junctions, which are grid-line endpoints terminating at the midpoint of an adjacent element edge. 
Following the notation established in \cite{Tian2011}, we define the PHT-spline space over {$\hat{\mathcal{T}}_h$ } as
\begin{equation}\label{eq:spline_space_S}
\mathcal{S}(p,q,\alpha,\beta, {\hat{\mathcal{T}}_h } ) := \Big\{ s \in C^{\alpha,\beta}(\hat{\Omega}) \mid s|_{T} \in \mathbb{P}_{p,q} \, \forall\, T \in {\hat{\mathcal{T}}_h } \Big\},
\end{equation}
where $\mathbb{P}_{p,q}$ denotes the space of bivariate tensor-product polynomials of bi-degree $(p,q)$ on each element, and $C^{\alpha,\beta}(\hat{\Omega})$ encodes the prescribed inter-element smoothness. 
For simplicity, we focus on the symmetric case $p=q$. 
To ensure global $C^1$ continuity and maintain the nestedness of the basis functions during local refinement, it is mathematically required that $p, q \ge 3$ \cite{Wang2011}.
In this work, we employ the spline space $\mathcal{S}(k,k,\alpha,\beta, \hat{\mathcal{T}}_h )$, where $k$ denotes the polynomial degree and we set $p = q = k \ge 3$, ensuring at least $C^1$ continuity with $\alpha, \beta \geq 1$.
Some intermediate estimates below are valid already for $k\ge2$; however, they are used only under this standing assumption and do not yield a $k=2$ version of the main results.

% --- Part 2: Basis Functions and Discrete Implementation ---
{ Following the standard hierarchical construction \cite{PHT-spline, Anitescu2018}, one obtains a set of global basis functions $\{R_a\}_{a=1}^{N_a} \subset \mathcal{S}(3,3,1,1,\hat{\mathcal{T}}_h)$, where
the dimension $N_a$ is directly associated with the number of crossing vertices and T-vertices. }
For efficient numerical implementation, these basis functions are constructed element-wise via Bézier extraction. 
Specifically, for each leaf element $T \in \hat{\mathcal{T}}_h$, there exists an extraction matrix $\mathbf{C}_T$ such that the restriction of the basis functions to $T$ is expressed as a linear combination of the local tensor-product Bernstein basis $\{B_j\}_{j=1}^{(k+1)^2}$ defined on the reference square $\hat{T} = [-1, 1]^2$. 
{Based on these, we define the discrete space as
\begin{align}\label{eq:Vh}
V_h := \left\{ 
v_h : \hat{\Omega} \rightarrow \mathbb{R}: v_h(\mathbf{x}) = \sum_{a=1}^{N_a} v_a R_a(\mathbf{x}), \,\, v_a \in \mathbb{R}
\right\}.
\end{align}
}

% --- Part 3: Numerical Advantages ---
The space $V_h$ inherits several critical properties for adaptive analysis. First, for $k\ge 3$, $V_h$ is globally $C^1$-continuous, which is particularly advantageous for Galerkin formulations involving second-order operators. Second, the hierarchical structure enables truly local mesh adaptivity while preserving the linear independence of the basis. Consequently, T-junctions do not introduce independent degrees of freedom (DOFs); instead, the extraction operator $\mathbf{C}_T$ implicitly enforces the necessary $C^1$ constraints across interfaces of differing refinement levels, ensuring a seamless coupling of the global solution.

\subsection{Unfitted mesh}
Let $\phi$ be a smooth level-set function, and let $\phi_h \in V_h$ be a high-order approximation of $\phi$. We define the discrete interface as $\Gamma_h := \{ \mathbf{x} : \phi_h(\mathbf{x}) = 0 \}$.
Let $\CT_h$ be the submesh of $\HCT_h$ obtained by 
discarding all elements whose closures lie entirely outside the discrete domain $\{\mathbf{x}: \phi_h(\mathbf{x}) < 0\}$. We define the sets of interior cells and cut (strip) cells by
\[
\begin{aligned}
\mathcal{T}_h^{\rm in} &:= \left\{T \in \CT_h : T \cap \Gamma_h = \varnothing \right\}, \\
\mathcal{T}_h^{\Gamma} &:= \left\{T \in \CT_h : T \cap \Gamma_h \neq \varnothing \right\},
\end{aligned}
\]
so that $\CT_h = \mathcal{T}_h^{\rm in} \cup \mathcal{T}_h^{\Gamma}$ and $\mathcal{T}_h^{\rm in} \cap \mathcal{T}_h^{\Gamma} = \varnothing$. Correspondingly, we introduce the open domains
\[
\Omega_h^{\rm in} := \Big(\bigcup_{T \in \mathcal{T}_h^{\rm in}} T\Big)^\circ, \quad \Omega_h^{\Gamma} := \Big(\bigcup_{T \in \mathcal{T}_h^{\Gamma}} T\Big)^\circ, \quad \text{and} \quad \Omega_h := \Big(\bigcup_{T \in \CT_h} T\Big)^\circ,
\]
which satisfy $\Omega_h^{\rm in} \cap \Omega_h^{\Gamma} = \varnothing$ and $\Omega_h = \left( \overline{\Omega_h^{\rm in}} \cup \overline{\Omega_h^{\Gamma}} \right)^\circ$.
To account for the overlapping supports of the spline basis functions, we further introduce an extended domain $\widetilde{\Omega}_h$, defined as the union of the supports of all basis functions that intersect $\Omega_h$, so that $\Omega_h \subset \widetilde{\Omega}_h \subset \hat{\Omega}$.

Let $\mathcal{F}_h^\Gamma$ denote the set of interior facets of $\CT_h$ that are either intersected by the discrete interface $\Gamma_h$ or are incident to a cut element. For the hierarchical T-mesh, we denote the diameter of a cell $T \in \hat{\mathcal{T}}_h$ by $h_T$, and define the minimum and maximum mesh sizes as
\[
h_m := \min_{T \in \hat{\mathcal{T}}_h} h_T, \qquad h_M := \max_{T \in \hat{\mathcal{T}}_h} h_T,
\]
respectively. 
Throughout the analysis, we set
$ h:=h_M=\max_{T\in\widehat{\mathcal T}_h}h_T$. 
Figure~\ref{Figure_Tmesh} illustrates the hierarchical T-mesh and the domain partitioning.

\begin{figure}[htbp]
    \centering
    \includegraphics[width=0.45\textwidth]{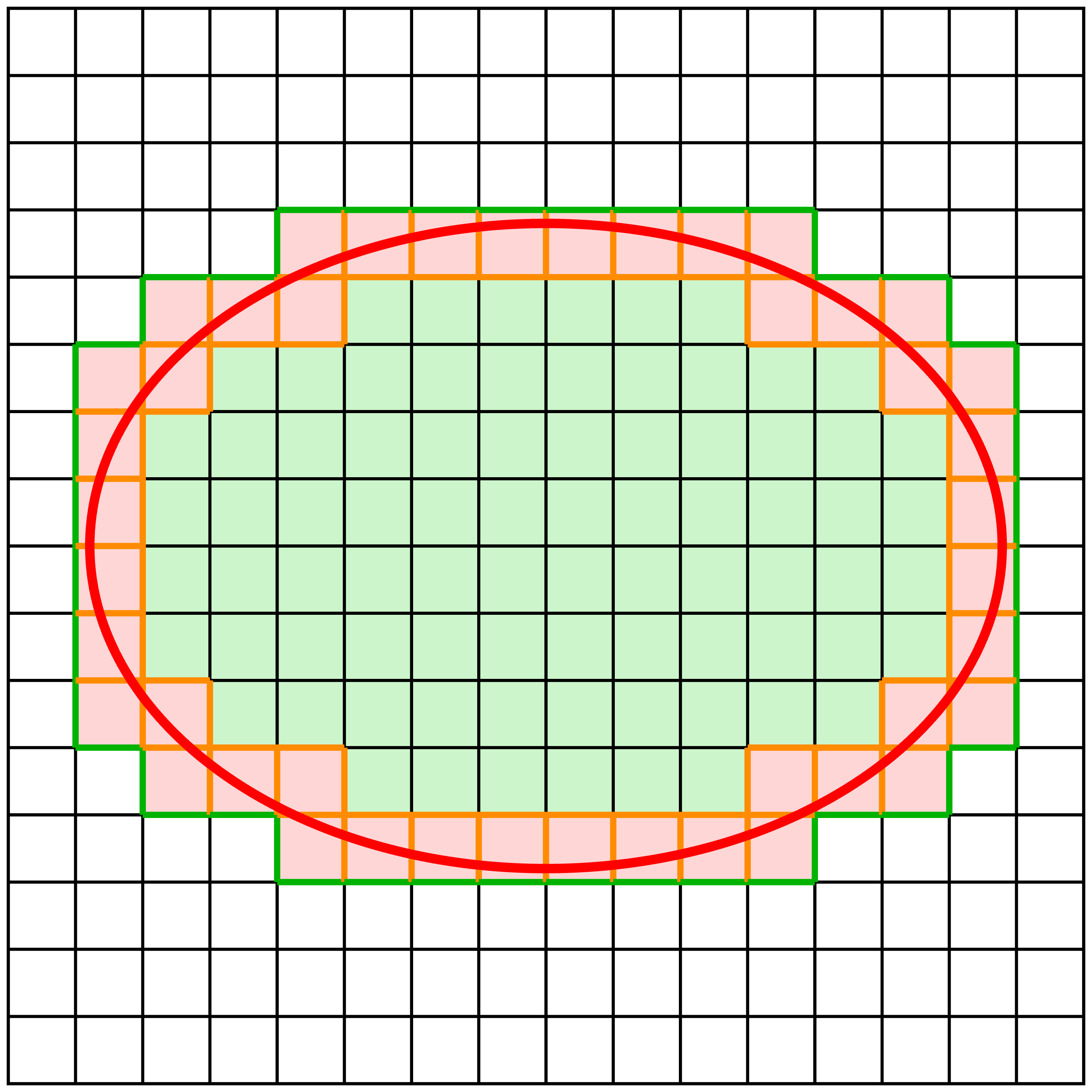}
    \hfill
    \includegraphics[width=0.45\textwidth]{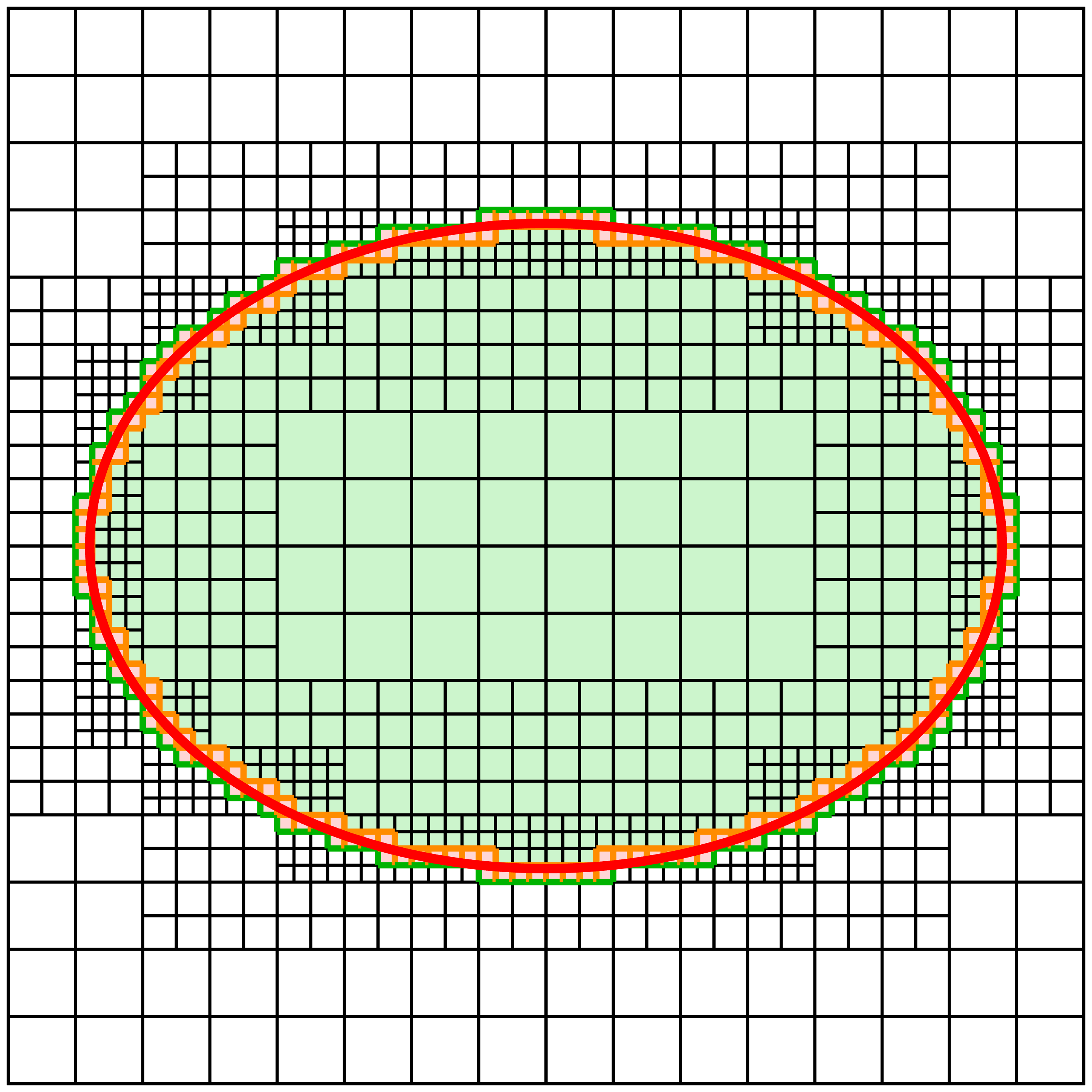}
    \caption{ Schematic of the computational domain $\hat{\Omega}$ and the associated background mesh $\HCT_h$. The green and pink cells represent the interior cells $\mathcal{T}_h^{\rm in}$ and cut cells $\mathcal{T}_h^\Gamma$, respectively. The solid red line denotes $\Gamma_h$. Furthermore, green and orange edges indicate the boundary edges in $\partial \Omega_h$ and the interface-associated facets in $\mathcal{F}_h^\Gamma$, respectively.}
    \label{Figure_Tmesh}
\end{figure}

\begin{assumption}\label{ass:phi} 
We assume that the level-set function $\phi$, its approximation $\phi_h \in V_h$, and the hierarchical T-mesh satisfy the following properties:
\begin{itemize}
\item[{\rm(P1)}]
$\phi \in C^{k+1}(\hat{\Omega})$ and the exact interface is given by $\Gamma := \{\mathbf{x} \in [0,1]^d : \phi(\mathbf{x}) = 0\}$. 
There exist constants $\delta_0,c_d,C_d>0$, independent of $h$, such that
\[
|\nabla\phi(\mathbf x)|\ge c_d,
\qquad
c_d\operatorname{dist}(\mathbf x,\Gamma)
\le |\phi(\mathbf x)|
\le C_d\operatorname{dist}(\mathbf x,\Gamma),
\]
for all $\mathbf x\in U_{\delta_0}(\Gamma)$.

\item[{\rm(P2)}] 
$\phi_h \in V_h$ satisfies the following approximation estimates:
\begin{align}
\| \phi - \phi_h \|_{L^2(\Omega_h)} &\le C h^{k+1} \| \phi \|_{H^{k+1}(\TOmega)}, \label{approx_phi_l2}\\
\| \phi - \phi_h \|_{H^1(\Omega_h)} &\le C h^k \| \phi \|_{H^{k+1}(\TOmega)}.\label{approx_phi_H1} 
\end{align}
In addition, the approximation is assumed to satisfy
\begin{align}
\|\phi-\phi_h\|_{L^\infty(\hat\Omega)}
+h\|\nabla(\phi-\phi_h)\|_{L^\infty(\hat\Omega)}
+h^2\|D_h^2(\phi-\phi_h)\|_{L^\infty(\hat\Omega)}
\le
C h^{k+1}\|\phi\|_{W^{k+1,\infty}(\hat\Omega)},
\label{approx_phi_max}
\end{align}
where $D_h^2$ denotes the elementwise second derivative.

\item[{\rm(P3)}] 
The cut-to-interior proximity condition holds:
there exists $c_b > 0$ such that for every $T \in \CT_h^\Gamma$, there exists an interior element $K \in \CT_h^{\rm in}$ satisfying $\text{dist}(T,K) \le c_b h_T$.

\item[{\rm(P4)}] 
The hierarchical mesh is 1-regular, meaning that the refinement levels of adjacent cells differ by at most one, ensuring that the mesh sizes of neighboring cells are comparable.
\end{itemize}
\end{assumption}
\begin{remark}\label{rmk:P2}\upshape
Property (P2) is a standard requirement in the context of unfitted finite element methods, which is naturally satisfied when {the $k$-th degree PHT-spline function $\phi_h \in V_h$ is defined as the global spline interpolation 
% or the $L^2$-projection 
of a sufficiently smooth level-set function $\phi$.}
Furthermore, for polynomial degrees $k \ge 3$, the $L^\infty$-norm of $\nabla\phi_h$ is uniformly bounded as a consequence of (P2) and the standard inverse inequalities for splines. 
\end{remark}

\begin{remark}\label{rmk:path}\upshape
Properties (P1)--(P2) ensure that the discrete interface $\Gamma_h$ provides a high-order approximation of $\Gamma$ while maintaining its non-degeneracy. 
Property (P3) serves as a standard geometric assumption, ensuring that every cut cell is sufficiently close to the interior to facilitate the stable extension of the solution. 
Furthermore, (P4) is critical for guaranteeing that the number of elements $N$ in a face-adjacent path from a cut cell to an interior cell remains uniformly bounded, independent of the local refinement level. 
This property is essential for establishing the uniform stability of the ghost penalty terms across non-conforming mesh interfaces.
\end{remark}

\subsection{A simple least-squares UnCut FEM using PHT-splines}
\label{sec:scheme}
In this subsection, we present the proposed numerical scheme and state the main theoretical results.
Throughout this paper, the notation $A \lesssim B$ denotes $A \le C B$ with a positive constant $C$ independent of the mesh size, while $A \approx B$ indicates that $A \lesssim B$ and $B \lesssim A$ simultaneously.
We utilize $(\cdot, \cdot)_{L^2(D)}$ and $(\cdot, \cdot)_{H^1(D)}$ to denote the standard $L^2$ and $H^1$ inner products over a domain $D$.

Based on the spline space $V_h$ and the approximate level-set function $\phi_h$, we define the finite element space as $S_h := \{ \phi_h v_h : v_h \in V_h \}$. 
Since the functions in $S_h$ are globally $C^1$-continuous, their zero-th and first-order normal jumps across any internal face are intrinsically zero. 
Thus, the jump penalty term only involves higher-order derivatives from $j=2$ to $2k$, defined as 
\begin{align}\label{penalty}
j_h(u_h, v_h) := {\sum_{j=2}^{2k} } \sum_{F \in \mathcal{F}_h^\Gamma} h_F^{2j-3} \int_{F} \jump{\partial^j_n u_h} \jump{\partial^j_n v_h}, \quad \forall\, (u_h, v_h) \in S_h \times S_h,
\end{align}
where $h_F := \min(h_{T^+}, h_{T^-})$ denotes the characteristic mesh size associated with the interior face $F = T^+ \cap T^-$.

It is important to emphasize that all computations in the proposed scheme \eqref{scheme} are carried out directly on the background mesh $\CT_h$. This approach circumvents the need for complex numerical integration over geometrically curved cut elements induced by the intersection of the physical boundary with the Cartesian grid. 
Consequently, the scheme relies solely on standard element-wise quadrature rules, significantly simplifying the implementation and the computational pipeline.

\begin{theorem}[$H^1$-error estimate]\label{Thm_H1}
Under Assumption \ref{ass:phi}, 
assuming $\Omega \subset \Omega_h$,
let $f \in H^{k-1}(\Omega_h)$ and let $\tilde{u} \in H^{k+1}(\TOmega)$ be the extension of the exact solution $u$, 
{satisfying $\Tu |_{\Omega} = u$ and $\| \Tu \|_{H^{k+1}(\TOmega)}
\leq C\|u\|_{H^{k+1}(\Omega)}$. }
Then the numerical scheme in \eqref{scheme} admits a unique solution $u_h \in S_h$, and the following error estimate holds:
\begin{equation}\label{H1_error}
\| \tilde{u} - u_h \|_{H^1(\Omega_h)} \leq C h^k \| f \|_{H^{k-1}(\Omega_h)}.
\end{equation}
\end{theorem}

Although the hierarchical T-mesh framework permits local refinement, 
the {\emph a priori} analysis in Section~\ref{sec:main} is carried out under the additional  assumption that the active meshes are globally comparable, namely,
$h_M / h_m \le C$ with a constant independent of $h$. 
Accordingly, we use the global mesh  parameter $h:=h_M$ throughout Section~\ref{sec:main}. 
The locally refined hierarchical T-meshes  considered in Section~\ref{sec:poster} are investigated from the viewpoint of adaptive  computation and are not required to satisfy this global mesh-ratio condition.

\section{Proof of Theorem \ref{Thm_H1}}
\label{sec:main}
In this section, we present a detailed error analysis of the proposed scheme \eqref{scheme}. We first establish several technical preliminaries, including stability and approximation results. 
We then prove the coercivity and continuity of the associated bilinear form, which ensures the well-posedness of the discrete problem. 
Finally, in Subsection~\ref{sec:main_H1_error}, we combine these results to derive the optimal \emph{a priori} error estimates.

\subsection{Technical results}
\label{sec:main_technical}

\begin{lemma}[{\cite[Lemma~3.6]{Duprez2020}}]\label{lem:vanish}
For any integer $s$ with $1 \le s \le k+1$ and any $v \in H^s(\Omega_h)$ that vanishes in $ \Omega$, it holds that
\begin{align}
\| v \|_{L^2(\Omega_h\backslash \Omega)} \leq C h^s \| v \|_{H^s(\Omega_h\backslash \Omega)}.
\end{align}
\end{lemma}

\begin{lemma}\label{lem:stability_H2}
For any $v_h \in S_h$,
and two neighboring elements $K^+$ and $K^-$ sharing $F\in\mathcal F_h^\Gamma$, we have the local estimate
\begin{equation}\label{Stability_H2}
\|v_h\|_{H^2(K^+)}^2
\le C\Big(\|v_h\|_{H^2(K^-)}^2+\sum_{j=2}^{2k} h_F^{2j-3}\| \jump{\partial_n^j v_h} \|_{L^2(F)}^2\Big).
\end{equation}
\end{lemma}

\begin{proof}
By affine changes of variables, it suffices to establish the estimate on each of the finitely many fixed reference macro-elements
$\widehat K=\widehat K^+\cup\widehat K^-$ 
arising from the 1-regular hierarchical T-mesh, 
where $\widehat K^+$ and $\widehat K^-$ share a common edge segment $\widehat F$. 
Let $\widehat n$ be a fixed unit normal to $\widehat F$.
We define the finite-dimensional broken polynomial space of degree $2k$ on $\widehat{K}$ as:
\[
X_k(\widehat{K}) := \{ \widehat{p} \in L^2(\widehat{K}) : \widehat{p}|_{\widehat{K}^\pm} \in \mathbb{Q}_{2k}(\widehat{K}^\pm) \}. 
\]
For any $\widehat{p} \in X_k(\widehat{K})$, consider the following functional:
\[
N(\widehat{p})^2 := \|\widehat{p}\|_{H^2(\widehat{K}^-)}^2 + \sum_{j=0}^{2k} \|\jump{\partial_{\widehat{n}}^j \widehat{p}}\|_{L^2(\widehat{F})}^2.
\]
We first show that $N(\cdot)$ defines a norm on $X_k(\widehat{K})$. Suppose $N(\widehat{p}) = 0$. The vanishing of the first term implies $\widehat{p} \equiv 0$ on $\widehat{K}^-$, which further ensures that all its derivative traces vanish on $\widehat{F}$ from the side of $\widehat{K}^-$. The vanishing of the jump terms $\jump{\partial_{\widehat{n}}^j \widehat{p}} = 0$ for $j=0, \dots, 2k$ then strictly forces $\partial_{\widehat{n}}^j \widehat{p}^+ = 0$ on $\widehat{F}$ for $j=0, \dots, 2k$. Since $\widehat{p}|_{\widehat{K}^+}$ 
has degree at most $2k$ in the direction normal to $\widehat F$,
and we have $2k+1$ homogeneous conditions on the normal derivatives along the edge $\widehat{F}$, the factor theorem implies that $\widehat{p}$ must be identically zero on $\widehat{K}^+$. Thus $N(\cdot)$ is a valid norm on $X_k(\widehat{K})$.

By the equivalence of norms in finite-dimensional spaces, $N(\cdot)$ dominates the $H^2(\widehat{K}^+)$-norm. 
For our specific discrete function $v_h \in S_h$, the global $C^1$-continuity implies $\jump{\partial_n^0 v_h} = 0$ and $\jump{\partial_n^1 v_h} = 0$, which allows the jump summation to effectively start from $j=2$.

Finally, using $h_{K^+}\simeq h_{K^-}\simeq h_F$ and the affine
scaling relations
\[
|v_h|_{H^m(K)}^2
\simeq
h_F^{d-2m}
|\widehat v_h|_{H^m(\widehat K)}^2,
\qquad
\left\|
\jump{\partial_n^jv_h}
\right\|_{L^2(F)}^2
\simeq
h_F^{d-1-2j}
\left\|
\jump{\partial_{\widehat n}^j\widehat v_h}
\right\|_{L^2(\widehat F)}^2,
\]
we obtain the factor $h_F^{2j-3}$ and hence
\eqref{Stability_H2}.
\end{proof}

\begin{lemma}[Stability property]\label{stability}
For any $v_h \in S_h$, 
we have the following stability estimate
\begin{equation}\label{stability_equ}
\| v_h \|^2_{H^2(\Omega_h)} 
\leq C \Big( \| \Delta v_h \|^2_{L^2(\Omega_h)} 
+ j_h(v_h, v_h) \Big).
\end{equation}
\end{lemma}

\begin{proof}
Let $K\in \CT_h^\Gamma$ be a cut element. 
By Remark \ref{rmk:path}, there exists a path of elements
\[
K=K_0,K_1,\dots,K_N
\]
such that $K_N\in \CT_h^{\rm in}$ and, for each $i=1,\dots,N$, 
the elements $K_{i-1}$ and $K_i$ share a common face 
$F_i\in\mathcal F_h^\Gamma$.
Applying Lemma~\ref{lem:stability_H2} successively along the path 
$\{K_i\}_{i=0}^N$, we obtain
\[
\|v_h\|_{H^2(K)}^2
\le C\Big(
\|v_h\|_{H^2(K_N)}^2
+\sum_{i=1}^N\sum_{j=2}^{2k}
h_{F_i}^{2j-3}
\| \jump{\partial_n^j v_h} \|_{L^2(F_i)}^2
\Big).
\]

% By the geometric approximation of $\Gamma$ by $\Gamma_h$, the above
% paths can be extended by a uniformly bounded number of elements so
% that their terminal elements are contained in $\Omega$. Applying
% Lemma~\ref{lem:stability_H2} successively along the extended paths,
% summing over the elements outside $\Omega$, and using the bounded
% overlap of the paths, we obtain
For $\mathbf{x}\in\Omega_h^{\mathrm{in}}\setminus\Omega$, Assumptions (P1)-(P2) give $0\le\phi(\mathbf{x})\le Ch^{k+1}$ which implies
$|\Omega_h^{\mathrm{in}}\setminus\Omega|\le Ch^{k+1}$.
The polynomial inverse estimate on the quasi-uniform mesh then yields
\[
\|v_h\|_{H^2(\Omega_h^{\mathrm{in}})}^2
\le
\|v_h\|_{H^2(\Omega)}^2
+
Ch^{k+1-d}\|v_h\|_{H^2(\Omega_h)}^2.
\]
Combining this bound with the preceding path estimates and using
their bounded overlap, we can absorb the resulting
$Ch^{k+1-d}\|v_h\|_{H^2(\Omega_h)}^2$ term for sufficiently small $h$,
since $k+1-d>0$. We therefore obtain 
\begin{equation}
\label{eq:H2-extension-Omega}
\|v_h\|_{H^2(\Omega_h)}^2
\le C\Big(
\|v_h\|_{H^2(\Omega)}^2
+j_h(v_h,v_h)
\Big).
\end{equation}

Write
\(
v_h=\phi_h w_h
=\phi w_h+(\phi_h-\phi)w_h.
\)
Since $\phi=0$ on
$\Gamma=\partial\Omega$, we have
$\phi w_h\in H^2(\Omega)\cap H_0^1(\Omega)$. The $H^2$-regularity of
the physical domain $\Omega$ therefore gives
\begin{align*}
\|\phi w_h\|_{H^2(\Omega)}
&\le C\|\Delta(\phi w_h)\|_{L^2(\Omega)}
\\
&\le C\|\Delta v_h\|_{L^2(\Omega)}
   +C\|\Delta ((\phi_h - \phi) w_h) \|_{L^2(\Omega)}.
\end{align*}
Consequently,
\begin{equation}
\label{eq:physical-H2-pre}
\|v_h\|_{H^2(\Omega)}
\le C\|\Delta v_h\|_{L^2(\Omega)}
   +C\| (\phi_h - \phi) w_h\|_{H^2(\Omega)}.
\end{equation}

Using the approximation properties of $\phi_h$, the product rule,
and the local inverse inequalities, we obtain
\begin{align*}
\| (\phi_h - \phi) w_h\|_{H^2(\Omega)}
\le C\Big(
&h^{k-1}\|w_h\|_{L^\infty(\Omega_h)}
+h^k\|w_h\|_{W^{1,\infty}(\Omega_h)}
\\
&+h^{k+1}\|w_h\|_{W^{2,\infty}(\Omega_h)}
\Big)
\le Ch^{k-1}\|w_h\|_{L^\infty(\Omega_h)}.
\end{align*}

By the Hardy-type inequality associated with the exact level-set function $\phi$,
\[
  \|w_h\|_{L^\infty(\Omega_h)}
  \le C\|\phi w_h\|_{W^{1,\infty}(\Omega_h)}.
\]
Since $\phi w_h = v_h - (\phi_h-\phi)w_h$, 
the product rule, the maximum-norm approximation properties of $\phi_h$, 
and the inverse inequality yield
\[
\|w_h\|_{L^\infty(\Omega_h)}
\le C\|v_h\|_{W^{1,\infty}(\Omega_h)}
+C\left( 
\|\nabla(\phi_h-\phi)\|_{L^\infty(\Omega_h)}
  +h^{-1}\|\phi_h-\phi\|_{L^\infty(\Omega_h)}
\right) \|w_h\|_{L^\infty(\Omega_h)}.
\]
For $h$ sufficiently small, the last term can be absorbed, and hence
\[
  \|w_h\|_{L^\infty(\Omega_h)}
  \le C\|v_h\|_{W^{1,\infty}(\Omega_h)}
  \le Ch^{-d/2}\|v_h\|_{H^2(\Omega_h)}.
\]
It follows that
\begin{equation}
\label{eq:geometric-small-term}
\|(\phi_h-\phi)w_h\|_{H^2(\Omega)}
\le
Ch^{k-1-d/2}\|v_h\|_{H^2(\Omega_h)}.
\end{equation}

Combining \eqref{eq:physical-H2-pre} and
\eqref{eq:geometric-small-term}, we obtain
\[
\|v_h\|_{H^2(\Omega)}
\le C\|\Delta v_h\|_{L^2(\Omega)}
   +Ch^{k-1-d/2}\|v_h\|_{H^2(\Omega_h)}.
\]
Substituting this estimate into \eqref{eq:H2-extension-Omega} and
taking square roots gives
\[
\|v_h\|_{H^2(\Omega_h)}
\le C\Big(
\|\Delta v_h\|_{L^2(\Omega)}
+j_h(v_h,v_h)^{1/2}
+h^{k-1-d/2}\|v_h\|_{H^2(\Omega_h)}
\Big).
\]
Since $k\ge3$, the last term can be absorbed into the left-hand side
for $h$ sufficiently small. Hence,
\begin{align*}
\|v_h\|_{H^2(\Omega_h)}^2
\le C\Big(
\|\Delta v_h\|_{L^2(\Omega)}^2
+j_h(v_h,v_h)
\Big) 
\le C\Big(
\|\Delta v_h\|_{L^2(\Omega_h)}^2
+j_h(v_h,v_h)
\Big)
\end{align*}
with $\Omega\subset\Omega_h$,
which is the desired estimate and concludes the proof. 
\hfill\end{proof}

\begin{corollary}
\label{cor:H1_extension}
For every $v_h\in S_h$,
\begin{equation}
\label{eq:H1_extension}
\|v_h\|_{H^1(\Omega_h)}^2
\le C\Big(
\|\nabla v_h\|_{L^2(\Omega)}^2
+\|v_h\|_{L^2(\Gamma)}^2
+h^2j_h(v_h,v_h)
\Big),
\end{equation}
where $C$ is independent of $h$.
\end{corollary}

\begin{proof}
Repeating the path argument in Lemma~\ref{stability}
for derivatives of order at most one gives
\[
\|v_h\|_{H^1(\Omega_h)}^2
\le C\Big(
\|v_h\|_{H^1(\Omega)}^2
+h^2j_h(v_h,v_h)
\Big).
\]
The generalized Poincar\'e--Friedrichs inequality on the
fixed domain $\Omega$ yields
\[
\|v_h\|_{H^1(\Omega)}^2
\le C\Big(
\|\nabla v_h\|_{L^2(\Omega)}^2
+\|v_h\|_{L^2(\Gamma)}^2
\Big).
\]
Combining these two estimates proves the result.
\end{proof}

\begin{lemma}[Interpolation with geometric error]\label{lem:interpolation}
Under Assumption \ref{ass:phi}, for any $v \in H^{k+1}(\TOmega)$ with $ v|_{\Omega} \in H_0^1(\Omega)$, let $z := v/\phi$ be the corresponding scaled function. 
Let $\mathcal{I}_h$ be the standard spline interpolation operator. We define the modified interpolant $\widetilde{\mathcal{I}}_h v \in S_h$ as
\begin{equation}\label{eq:def_interpolant}
\widetilde{\mathcal{I}}_h v := \phi_h (\mathcal{I}_h z).
\end{equation}
Then, the following error estimate holds for $m \in \{1, 2\}$:
\begin{equation}\label{eq:interp_H1_H2}
\| v - \widetilde{\mathcal{I}}_h v \|_{H^m(\Omega_h)}
\le C h^{k+1-m} \| v \|_{H^{k+1}(\TOmega)} ,
\end{equation}
where $C$ may depend on higher-order Sobolev norms of the level set function $\phi$.
\end{lemma}

\begin{proof}
Denote $z_h = \CI_h z$ as the spline interpolant of $z$. 
We decompose the interpolation error as
\begin{equation}\label{split_phi}
\phi z-\phi_h z_h
= \underbrace{(\phi-\phi_h)z_h}_{T_1}
+ \underbrace{\phi(z-z_h)}_{T_2}
% - \underbrace{(\phi-\phi_h)(z-z_h)}_{T_3},
\end{equation}
and estimate $T_1$ and $T_2$ in the $H^1(\Omega_h)$- and $H^2(\Omega_h)$-norms, respectively.

\textbf{Estimate for $T_1$:} Using the Sobolev embeddings $H^2(\Omega_h)\hookrightarrow L^\infty(\Omega_h)$ and $H^1(\Omega_h)\hookrightarrow L^6(\Omega_h) \hookrightarrow L^3(\Omega_h)$, we obtain
\begin{align}
\|T_1\|_{H^1(\Omega_h)}
&\le C\Big(
\|(\phi-\phi_h)z_h\|_{L^2(\Omega_h)}
+\|(\nabla(\phi-\phi_h))\,z_h\|_{L^2(\Omega_h)}
+\|(\phi-\phi_h)\nabla z_h\|_{L^2(\Omega_h)}
\Big) \notag \\
&\le C\Big(
\|\phi-\phi_h\|_{L^2(\Omega_h)}\|z_h\|_{L^\infty(\Omega_h)}
+ \|\nabla(\phi-\phi_h)\|_{L^2(\Omega_h)}\|z_h\|_{L^\infty(\Omega_h)} \notag \\
&\qquad + \|\phi-\phi_h\|_{L^3(\Omega_h)}\|\nabla z_h\|_{L^6(\Omega_h)}
\Big) \notag \\
&\le C\,\|\phi-\phi_h\|_{H^1(\Omega_h)}\,\|z_h\|_{H^2(\Omega_h)} \le C\,h^k\,\|z\|_{H^2(\TOmega)},
\label{eq:T1_final_H1_temp}
\end{align}
where the last inequality follows from the $H^2$-stability of the spline interpolation operator and the approximation property \eqref{approx_phi_H1}.
By the Hardy-type result \cite[Lemma~3.1]{Duprez2020},
we have $z\in H^{k}(\TOmega)$ and
\begin{align}\label{eq:hardy_z}
\|z\|_{H^{k}(\TOmega)} 
\le C \|v\|_{H^{k+1}(\TOmega)}.
\end{align}
Therefore, substituting the Hardy-type inequality \eqref{eq:hardy_z} into \eqref{eq:T1_final_H1_temp}, we obtain
\begin{align}
\|T_1\|_{H^1(\Omega_h)}
\le C h^k \|v\|_{H^{3}(\TOmega)}.
\label{eq:T1_final_H1}
\end{align}
Similarly, the $H^2$-norm admits the estimate
\begin{align}
\|T_1\|_{H^2(\Omega_h)}
% &\le C\Big(
% \|(\phi-\phi_h)z\|_{L^2(\Omega_h)}
% +\|\nabla(\phi-\phi_h)\,z\|_{L^2(\Omega_h)}
% +\|\nabla^2(\phi-\phi_h)\,z\|_{L^2(\Omega_h)} \notag \\
% &\qquad
% +\|(\phi-\phi_h)\nabla z\|_{L^2(\Omega_h)}
% +\|\nabla(\phi-\phi_h)\cdot\nabla z\|_{L^2(\Omega_h)}
% +\|(\phi-\phi_h)\nabla^2 z\|_{L^2(\Omega_h)}
% \Big) \notag \\
&\le C\Big(\|\phi-\phi_h\|_{H^2(\Omega_h)}\|z_h\|_{L^\infty(\Omega_h)}
+ \|\nabla(\phi-\phi_h)\|_{L^3(\Omega_h)}\|\nabla z_h\|_{L^6(\Omega_h)}\notag\\
& \qquad + \|\phi-\phi_h\|_{L^\infty(\Omega_h)}\|z_h\|_{H^2(\Omega_h)}\Big) \notag \\
&\le C\,\|\phi-\phi_h\|_{H^2(\Omega_h)}\,\|z_h\|_{H^2(\Omega_h)} \quad \text{(using Sobolev embedding)} \notag \\
&\le C\,\|\phi-\phi_h\|_{H^2(\Omega_h)}\,\|v\|_{H^{3}(\TOmega)},
\label{eq:T1_final_H2_temp}
\end{align}
where the last inequality uses \eqref{eq:hardy_z} and the $H^2$-stability of the spline interpolation operator.
Let $\CI_h \phi \in V_h$ be the interpolation of $\phi$. Since $\CI_h \phi - \phi_h \in V_h$, the inverse inequality is applicable. Moreover, by Assumption~\ref{ass:phi} (P2) and the standard interpolation estimate,
\[
\|\phi - \CI_h \phi\|_{H^2(\Omega_h)} \le C h^{k-1}\|\phi\|_{H^{k+1}(\TOmega)}
\qquad \text{for } k\ge 2.
\]
Therefore,
\begin{align}\label{H2-approx}
\| \phi - \phi_h \|_{H^2(\Omega_h)}
&\le \| \phi - \CI_h\phi \|_{H^2(\Omega_h)}
+ \| \CI_h \phi - \phi_h \|_{H^2(\Omega_h)}\notag\\
&\le h^{k-1}\| \phi \|_{H^{k+1}(\TOmega)} + h^{-1}\| \CI_h \phi - \phi_h \|_{H^1(\Omega_h)}\notag\\
&\le h^{k-1}\| \phi \|_{H^{k+1}(\TOmega)} 
+ h^{-1}\Big( \| \CI_h \phi - \phi \|_{H^1(\Omega_h)}
+ \| \phi - \phi_h \|_{H^1(\Omega_h)} \Big)\notag\\
&\le C h^{k-1},
\end{align}
where $C$ depends on $\|\phi\|_{H^{k+1}(\TOmega)}$. Substituting \eqref{H2-approx} into \eqref{eq:T1_final_H2_temp}, we arrive at
\begin{align}\label{T1-H2-final}
    \|T_1\|_{H^2(\Omega_h)} \le C h^{k-1} \|v\|_{H^{3}(\TOmega)} .
\end{align}

\textbf{Estimate for $T_2$:}
For $T_2 = \phi(z - z_h)$, we use the Leibniz rule to estimate the $H^m$-norm for $m \in \{1,2\}$ in a unified way. Since $\phi \in C^{k+1}(\hat{\Omega})$, its derivatives are uniformly bounded, and hence
\begin{equation}\label{eq:T2_bound_simplified}
\|\phi (z-z_h)\|_{H^m(T)}
\le C \Big( \|\phi\|_{L^\infty(T)} \|z-z_h\|_{H^m(T)} + \|z-z_h\|_{H^{m-1}(T)} \Big).
\end{equation}
The second term on the right-hand side can be handled directly. 
Indeed, by the standard local approximation property of the
PHT-spline interpolation operator, we have
\begin{align*}
\| z-z_h \|_{H^{m-1}(T)}
&\le Ch^{k-m+1} \| z \|_{H^k(\omega_T)}.
\end{align*} 
Using the uniformly bounded overlap of the patches $\omega_T$ and
the Hardy-type inequality \eqref{eq:hardy_z}, we obtain
\begin{align}\label{lagrange-inter-low}
\sum_{T\in\mathcal T_h}
\| z-z_h \|_{H^{m-1}(T)}^2
\le
Ch^{2(k-m+1)}
\|z\|_{H^k(\widetilde\Omega_h)}^2
\le
Ch^{2(k-m+1)}
\|v\|_{H^{k+1}(\widetilde\Omega_h)}^2.
\end{align} 
The first term, however, is more delicate. 
A direct application of the interpolation estimate would involve the quantity $\|z\|_{H^{k+1}(\TOmega)}$, and using the Hardy-type inequality in a straightforward manner would then require higher regularity of $v$ than desired. To avoid this loss of derivatives, we proceed with a more refined argument. 
Following the approach in \cite[Lemma~6]{phifem_stokes} and \cite[Lemma~3.5]{HuLiYang}, we analyze \eqref{eq:T2_bound_simplified} by decomposing the mesh $\mathcal{T}_h$ into two disjoint subsets according to the distance of each element from the interface $\Gamma$. 
Here, $\omega_T$ denotes the element patch associated with the local interpolation on $T$.

\noindent \textbf{Case 1: Interior domain} $\{ T \in \mathcal{T}_h : \operatorname{dist}(\omega_T, \Gamma) \ge 3h \}$. \\
For any $T$ in this set, the level set function $\phi$ does not vanish on the patch $\omega_T$. Since $\phi$ is Lipschitz continuous and satisfies $\phi = 0$ on $\Gamma$, there exists a constant $c_0>0$ such that
\[
\min_{x \in \omega_T} |\phi(x)| \ge c_0 h.
\]
Moreover, the standard interpolation estimate on $T$ gives
\[
\|z-z_h\|_{H^l(T)} \le C h^{k+1-l} |z|_{H^{k+1}(\omega_T)},
\qquad l \le m.
\]
Applying \cite[(33)]{phifem_stokes} on the patch $\omega_T$, we obtain
$$
|z|_{H^{k+1}(\omega_T)} \le \frac{C}{\min_{\omega_T}|\phi|}\Big(|v|_{H^{k+1}(\omega_T)}+\|z\|_{H^k(\omega_T)}\Big) .
$$
To control the prefactor involving $\phi$, we note that
\[
\frac{\|\phi\|_{L^\infty(T)}}{\min_{\omega_T}|\phi|}
=1+\frac{\|\phi\|_{L^\infty(T)}-\min_{\omega_T}|\phi|}{\min_{\omega_T}|\phi|}
\le 1+\frac{C h\|\nabla\phi\|_{L^\infty(T)}}{c_0 h} \le \tilde{C},
\]
where we have used the Lipschitz continuity of $\phi$. Therefore, combining the above estimates, we arrive at
\begin{align}
\|\phi\|_{L^\infty(T)} \|z-z_h\|_{H^m(T)}
&\le C h^{k+1-m} \|\phi\|_{L^\infty(T)} \|z\|_{H^{k+1}(\omega_T)} \notag \\
&\le C  h^{k+1-m}
\Big(\|v\|_{H^{k+1}(\omega_T)}+\|z\|_{H^k(\omega_T)}\Big).
% \notag\\
% &\le C  h^{k+1-m}
% \|v\|_{H^{k+1}(\omega_T)},
\label{eq:interior_phi_e}
\end{align}
% where the last inequality follows from the Hardy-type inequality \eqref{eq:hardy_z}.

\noindent \textbf{Case 2: Strip domain} $\{ T \in \mathcal{T}_h : \dist(\omega_T, \Gamma) < 3h \}$. \\
For any $T$ in this set, the patch $\omega_T$ lies within a strip of width $\mathcal{O}(h)$ around the interface $\Gamma$. Since $\phi = 0$ on $\Gamma$ and $\|\nabla \phi\|_{L^\infty(\hat{\Omega})}$ is bounded, it follows that
\(
\|\phi\|_{L^\infty(T)} \lesssim h.
\)
This provides an additional factor of $h$, and therefore the standard approximation estimate for the PHT-spline interpolation is sufficient to control the term $\|\phi\|_{L^\infty(T)} \|z-z_h\|_{H^m(T)}$ using only $\|z\|_{H^k(\omega_T)}$. Consequently,
% \begin{align}
% \|\phi\|_{L^\infty(T)} \|z-z_h\|_{H^m(T)}
% &\le C  h \cdot h^{k-m} \|z\|_{H^k(\omega_T)} \le C h^{k+1-m}
% \|v\|_{H^{k+1}(\omega_T)}.
% \label{eq:strip_phi_e}
% \end{align}
\begin{align}
\|\phi\|_{L^\infty(T)} \|z-z_h\|_{H^m(T)}
&\le C  h \cdot h^{k-m} \|z\|_{H^k(\omega_T)}.
% \le C h^{k+1-m} \|v\|_{H^{k+1}(\omega_T)}.
\label{eq:strip_phi_e}
\end{align}

By summing the estimates \eqref{lagrange-inter-low}--\eqref{eq:strip_phi_e} over all elements $T \in \mathcal{T}_h$, for $m \in \{1,2\}$, 
and then applying the Hardy-type inequality to $z=v/\phi$ on $\TOmega$,
we obtain the optimal estimates 
\[
\|T_2\|_{H^1(\Omega_h)} \le C h^k \|v\|_{H^{k+1}(\TOmega)},
\qquad
\|T_2\|_{H^2(\Omega_h)} \le C h^{k-1} \|v\|_{H^{k+1}(\TOmega)}.
\]
Finally, combining the above bounds for $T_2$ with those already derived for $T_1$, we conclude the proof of \eqref{eq:interp_H1_H2}.
\hfill\end{proof}

By arguing as in the proof of Lemma~\ref{lem:interpolation},
using the local approximation property of the PHT-spline interpolation operator on each element and standard inverse inequalities, 
one obtains the corresponding interpolation estimates.
Passing from these element-wise estimates to the global bound in the broken Sobolev norm is then justified by the uniformly bounded overlap of the patch family $\{\omega_T\}_{T\in\mathcal T_h}$ and the
Hardy-type inequality.
Therefore, the following corollary is immediate, and its proof is omitted.

\begin{corollary}\label{cor:interp_Hj}
Let $3\le r\le k+1$ be an integer and assume that
$v \in H^{k+1}(\TOmega)$ with $ v|_{\Omega} \in H_0^1(\Omega)$.
Under the geometric and mesh assumptions of
Lemma~\ref{lem:interpolation}, 
for $1\le j\le \min\{k,r\}$,
\begin{equation}\label{interpolation_Hj}
\|v-\widetilde{\mathcal I}_h v\|_{H_h^j(\Omega_h)}
\le Ch^{r-j}\|v\|_{H^r(\widetilde\Omega_h)},
\end{equation}
where the broken Sobolev norm is defined by
\begin{equation}\label{eq:def_H_h^j}
\|w\|_{H_h^j(\Omega_h)}^2:=\sum_{T \in \CT_h}\|w\|_{H^j(T)}^2.
\end{equation} 
\end{corollary}

\subsection{Stability and well-posedness}
\label{sec:main_stability}

We now establish the well-posedness of the discrete scheme \eqref{scheme}. To this end, we introduce the bilinear form $a_h(\cdot,\cdot)$ and the linear functional $L_h(\cdot)$ on $S_h$ by
\[
a_h(u_h,v_h)
:= (\Delta u_h,\Delta v_h)_{L^2(\Omega_h)} + j_h(u_h,v_h),
\qquad
L_h(v_h):=-(f,\Delta v_h)_{L^2(\Omega_h)}.
\]
By \eqref{stability_equ}, the bilinear form $a_h(\cdot,\cdot)$ is symmetric and positive definite on $S_h$. 
{ It therefore defines an inner product on $S_h$, with the induced energy norm given by
\[
\TN{v_h}^2
:= a_h(v_h,v_h)
= \|\Delta v_h\|_{L^2(\Omega_h)}^2 + j_h(v_h,v_h).
\] }
% Since $a_h(\cdot,\cdot)$ is symmetric and positive definite on $S_h$, the quantity
% \[
% \TN{v_h}^2 := a_h(v_h,v_h)
% \]
% defines a norm on $S_h$. 
Since $S_h$ is finite-dimensional, it follows that $(S_h,\TN{\cdot})$ is a Hilbert space. 
On the other hand, when extended to the larger space $H^2(\Omega_h)$, the quantity $\TN{\cdot}$ is in general only a semi-norm. 
Moreover, the linear functional $L_h(\cdot)$ is continuous on $S_h$, since, by the Cauchy--Schwarz inequality,
\[
|L_h(v_h)|
\le \|f\|_{L^2(\Omega_h)}\,\|\Delta v_h\|_{L^2(\Omega_h)}
\le \|f\|_{L^2(\Omega_h)}\,\TN{v_h}.
\]
Therefore, by the Riesz representation theorem, the variational problem \eqref{scheme} admits a unique solution $u_h\in S_h$.

\subsection{$H^1$-error estimate}
\label{sec:main_H1_error}
Before presenting the detailed proof of Theorem~\ref{Thm_H1}, we briefly outline the main strategy. The proof proceeds in three main steps: first, we establish an error estimate in the energy norm; next, we derive the $H^2$-error estimate; and finally, we recover the optimal $H^1$-error bound via a duality argument.

\begin{proof}
{ Recall that $\Tu$ is an extension of $u$ satisfying $ \| \Tu \|_{H^{k+1}(\TOmega)}
\leq C\|u\|_{H^{k+1}(\Omega)}$ and $\Omega \subset \Omega_h$. 
Combining this extension estimate with the elliptic regularity for the source problem $-\Delta u = f$ in $\Omega$, we obtain the following bound:
\begin{equation}\label{equ:u_f}
\| \Tu \|_{H^{k+1}(\TOmega)}
\leq C\|u\|_{H^{k+1}(\Omega)} 
\leq C\|f\|_{H^{k-1}(\Omega)}
\leq C\|f\|_{H^{k-1}(\Omega_h)}. 
\end{equation}
}

We define the residual $r := -\Delta \Tu - f$ in the extended domain $\Omega_h$. Since $\Tu = u$ in $\Omega$, it follows that the residual vanishes identically in $\Omega$. 
Furthermore, because $\Tu \in H^{k+1}(\Omega_h)$ and $f \in H^{k-1}(\Omega_h)$, we have $r \in H^{k-1}(\Omega_h)$. Therefore, by applying Lemma \ref{lem:vanish} to the residual $r$ and invoking the estimate \eqref{equ:u_f}, we obtain
\begin{align}
\| r\|_{L^2(\Omega_h)}
= \| r\|_{L^2(\Omega_h\setminus \Omega)}
&\leq C h^{k-1} \| r\|_{H^{k-1}(\Omega_h\setminus \Omega)} \notag \\
&\leq C h^{k-1} \Big(\| \Delta \Tu \|_{H^{k-1}(\Omega_h\setminus \Omega)} + \| f\|_{H^{k-1}(\Omega_h\setminus \Omega)} \Big) \notag \\
&\leq C h^{k-1} \Big(\| \Tu \|_{H^{k+1}(\Omega_h)} + \| f \|_{H^{k-1}(\Omega_h)} \Big) \notag \\
&\leq C h^{k-1} \| f \|_{H^{k-1}(\Omega_h)} . \label{equ:r_f}
\end{align}

Let $\TCI_h\Tu \in S_h$ be the spline interpolant defined in Lemma \ref{lem:interpolation}. 
We decompose the finite element error $e_h := \Tu - u_h$ into the interpolation error $\rho := \Tu - \TCI_h\Tu$ and the discrete error $\xi_h = \TCI_h\Tu - u_h \in S_h$:
\begin{equation}\label{decomp}
e_h = (\Tu - \TCI_h\Tu) + (\TCI_h\Tu - u_h) = \rho + \xi_h.
\end{equation}
We now derive the error equation for the discrete component $\xi_h$. 
From the definition of the bilinear form $a_h(\cdot,\cdot)$ and the discrete scheme \eqref{scheme}, for any test function $v_h \in S_h$, we have:
\begin{align}\label{discrete_residual}
a_h(\xi_h, v_h) 
&= a_h(\TCI_h\Tu, v_h) - a_h(u_h, v_h) \notag \\
&= (\Delta \TCI_h\Tu, \Delta v_h)_{L^2(\Omega_h)} + j_h(\TCI_h\Tu, v_h) + (f, \Delta v_h)_{L^2(\Omega_h)}.
\end{align}
Recall the continuous identity $-\Delta \Tu = f + r$ in the extended domain $\Omega_h$, which yields $f = -\Delta \Tu - r$. Substituting this relation into \eqref{discrete_residual}, we obtain
\begin{align}
a_h(\xi_h, v_h) &= (\Delta \TCI_h\Tu - \Delta \Tu, \Delta v_h)_{L^2(\Omega_h)} - (r, \Delta v_h)_{L^2(\Omega_h)} + j_h(\TCI_h\Tu, v_h) \notag \\
&= -(\Delta \rho, \Delta v_h)_{L^2(\Omega_h)} - (r, \Delta v_h)_{L^2(\Omega_h)} + j_h(\TCI_h\Tu, v_h).
\label{error_equ_new}
\end{align}

Choosing $v_h = \xi_h \in S_h$ in \eqref{error_equ_new}, we have  
\begin{equation}\label{energy_theta}
\TN{\xi_h}^2 = -(\Delta \rho, \Delta \xi_h)_{L^2(\Omega_h)} - (r, \Delta \xi_h)_{L^2(\Omega_h)} + j_h(\TCI_h\Tu, \xi_h).
\end{equation}

We bound the three terms on the right-hand side using the Cauchy--Schwarz inequality and the definition of the energy norm:
\begin{align*}
|(\Delta \rho, \Delta \xi_h)_{L^2(\Omega_h)}| 
&\le \|\Delta \rho\|_{L^2(\Omega_h)} \|\Delta \xi_h\|_{L^2(\Omega_h)} 
\le \|\Delta \rho\|_{L^2(\Omega_h)} \TN{\xi_h}, \\
|(r, \Delta \xi_h)_{L^2(\Omega_h)}| 
&\le \|r\|_{L^2(\Omega_h)} \|\Delta \xi_h\|_{L^2(\Omega_h)} \le \|r\|_{L^2(\Omega_h)} \TN{\xi_h}, \\
|j_h(\TCI_h\Tu, \xi_h)| &\le j_h^{1/2}(\TCI_h\Tu, \TCI_h\Tu) j_h^{1/2}(\xi_h, \xi_h) \le j_h^{1/2}(\TCI_h\Tu, \TCI_h\Tu) \TN{\xi_h}.
\end{align*}
Substituting these bounds into \eqref{energy_theta} and dividing both sides by $\TN{\xi_h}$, we obtain the stability estimate for the discrete component:
\begin{equation}\label{xi_bound_final}
\TN{\xi_h} \le \|\Delta \rho\|_{L^2(\Omega_h)} + \|r\|_{L^2(\Omega_h)} + j_h^{1/2}(\TCI_h\Tu, \TCI_h\Tu).
\end{equation}

To bound the right-hand side of \eqref{xi_bound_final}, we estimate each term separately. 
First, by the standard approximation properties of the spline interpolant given in Lemma \ref{lem:interpolation}, the Laplacian of the interpolation error is bounded by
\begin{equation}\label{eq:bound_rho}
\|\Delta\rho\|_{L^2(\Omega_h)}^2 
\le C \|\rho\|_{H^2(\Omega_h)}^2 \le C h^{2k-2} \|\Tu\|_{H^{k+1}(\TOmega)}^2,
\end{equation}
and the residual term is known to satisfy the estimate $\|r\|_{L^2(\Omega_h)}^2 \le C h^{2k-2} \|\Tu\|_{H^{k+1}(\TOmega)}^2$.

We split the summation in $j_h(\TCI_h\Tu, \TCI_h\Tu)$ into low-order ($2 \le j \le k$) and high-order ($k+1 \le j \le 2k$) contributions. 
For the low-order jumps, since the extended exact solution $\Tu \in H^{k+1}(\TOmega)$, the traces of its normal derivatives are continuous across internal faces, meaning $\jump{\partial_n^j \Tu} \equiv 0$. Thus, we can insert $\Tu$ to obtain $\jump{\partial_n^j \TCI_h \Tu} = \jump{\partial_n^j (\TCI_h \Tu - \Tu)} = -\jump{\partial_n^j \rho}$. Applying the standard continuous scaled trace inequality and the approximation properties of $\rho$, we have
\begin{align}\label{eq:penalty_low}
\sum_{j=2}^k \sum_{F\in \mathcal{F}_h^\Gamma} h_F^{2j-3} \|\jump{\partial_n^j \TCI_h \Tu}\|_{L^2(F)}^2 
&= \sum_{j=2}^k \sum_{F\in \mathcal{F}_h^\Gamma} h_F^{2j-3} \|\jump{\partial_n^j \rho}\|_{L^2(F)}^2 \notag \\
&\le C \sum_{j=2}^k h^{2j-3} \Big( h^{-1} h^{2(k+1-j)} \Big) \|\Tu\|_{H^{k+1}(\TOmega)}^2 \notag \\
&\le C h^{2k-2} \|\Tu\|_{H^{k+1}(\TOmega)}^2.
\end{align}

For the high-order jumps, the exact solution lacks sufficient regularity but we can directly apply the discrete trace inequality followed by the inverse inequality to reduce the order of derivatives from $j$ down to $k+1$, 
\begin{align}\label{eq:penalty_high}
\sum_{j=k+1}^{2k} \sum_{F\in \mathcal{F}_h^\Gamma} h_F^{2j-3} \|\jump{\partial_n^j \TCI_h \Tu}\|_{L^2(F)}^2 
&\le C\sum_{j=k+1}^{2k} \sum_{F\in \mathcal{F}_h^\Gamma} h^{2j-3} \Big( h^{2(k+1-j)-1} \|\TCI_h \Tu\|_{H^{k+1}(K)}^2 \Big) \notag \\
% &\le C h^{2k-2} \|\TCI_h \Tu\|_{H^{k+1}(\Omega_h)}^2 \notag \\
&\le C h^{2k-2} \|\Tu\|_{H^{k+1}(\TOmega)}^2
\end{align}
where we use the $H^{k+1}$-stability of the spline interpolant $\|\TCI_h\Tu\|_{H^{k+1}_h(\Omega_h)} \le C \|\Tu\|_{H^{k+1}(\TOmega)}$. 

Combining \eqref{eq:penalty_low} and \eqref{eq:penalty_high}, the total penalty term satisfies 
\begin{equation}\label{eq:penalty_Ihu}
j_h(\TCI_h\Tu, \TCI_h\Tu) \le C h^{2k-2} \|\Tu\|_{H^{k+1}(\TOmega)}^2.
\end{equation}
Finally, substituting this bound along with \eqref{eq:bound_rho} into \eqref{xi_bound_final}, we achieve the optimal energy norm estimate for $\xi_h$:
\begin{equation}\label{Energy_norm_xi}
\TN{\xi_h}^2 
\le C h^{2k-2} \|\Tu\|_{H^{k+1}(\TOmega)}^2 
\le C h^{2k-2} \|f\|_{H^{k-1}(\Omega_h)}^2,
\end{equation}
where the last step follows from \eqref{equ:u_f}.

We exploit the error decomposition $e_h = \rho + \xi_h$ and apply the triangle inequality:
\begin{equation}\label{eq:triangle_H2}
\|e_h\|_{H^2(\Omega_h)} 
\le \|\rho\|_{H^2(\Omega_h)} + \|\xi_h\|_{H^2(\Omega_h)}
\le C h^{k-1} \|\Tu\|_{H^{k+1}(\TOmega)} + \|\xi_h\|_{H^2(\Omega_h)}
\end{equation}
with the interpolation error estimate \eqref{eq:interp_H1_H2}.

By Lemma~\ref{stability}, the discrete error component $\xi_h \in S_h$ satisfies
$$
\|\xi_h\|_{H^2(\Omega_h)}^2 
\le C \Big( \|\Delta \xi_h\|_{L^2(\Omega_h)}^2 + j_h(\xi_h, \xi_h) \Big) = C \TN{\xi_h}^2.
$$
Consequently, we obtain the final optimal $H^2$-norm error estimate:
\begin{equation}\label{H2_Error}
\|e_h\|_{H^2(\Omega_h)} 
\le C h^{k-1} \|\Tu\|_{H^{k+1}(\TOmega)} 
\le C h^{k-1} \|f\|_{H^{k-1}(\Omega_h)}.
\end{equation}

We now establish the optimal $H^1$-error estimate by employing a standard duality argument. Let $v$ be the solution to the dual problem:
\begin{equation}\label{eq:dual_problem}
\left\{
\begin{aligned}
-\Delta v &= -e_h &&\text{in } \Omega, \\
v &= 0 &&\text{on } \Gamma.
\end{aligned}
\right.
\end{equation}
Extending the dual solution $v$ from $\Omega$ to the extended domain $\TOmega$, 
elliptic regularity on the fixed domain $\Omega$, followed by boundedness of this extension, gives 
\begin{align}
\label{eq:v_H3_stability}
\| v \|_{H^3(\TOmega)} \leq C \| v \|_{H^3(\Omega)} \leq C \| e_h \|_{H^1(\Omega_h)}.
\end{align}

Testing the numerical error $e_h$ against the Laplacian of the dual solution $\Delta v$ over the physical domain $\Omega$, and noting that $\Delta v = e_h$ from \eqref{eq:dual_problem}, we deduce:
\begin{align*}
-\int_\Omega \Delta e_h e_h
&= -\int_\Omega \Delta e_h \Delta v \notag \\
&= -(\Delta e_h, \Delta v)_{L^2(\Omega_h)} + (\Delta e_h, \Delta v)_{L^2({\Omega_h \setminus \Omega})} \notag \\
&= -(\Delta e_h, \Delta v_h)_{L^2(\Omega_h)} - (\Delta e_h, \Delta (v - v_h))_{L^2(\Omega_h)} + (\Delta e_h, \Delta v)_{L^2({\Omega_h \setminus \Omega})} \notag \\
&= (r, \Delta v_h)_{L^2(\Omega_h)} - (\Delta e_h, \Delta (v - v_h))_{L^2(\Omega_h)} + (\Delta e_h, \Delta v)_{L^2({\Omega_h \setminus \Omega})} \notag \\
&\quad - \big(j_h(\TCI_h \Tu, v_h) - j_h(\xi_h, v_h) \big),
\end{align*}
where the last equality follows from the error equation \eqref{error_equ_new}, and $v_h := \TCI_h v \in S_h$ denotes the modified interpolant defined in~\eqref{eq:def_interpolant}.
By Corollary~\ref{cor:interp_Hj}, it satisfies
$\|v-v_h\|_{H_h^3(\Omega_h)}
  \le C\|v\|_{H^3(\widetilde\Omega_h)}.
$
Hence, by the triangle inequality,
\begin{equation}
  \|v_h\|_{H_h^3(\Omega_h)}
  \le C\|v\|_{H^3(\widetilde\Omega_h)}
  \le C\|e_h\|_{H^1(\Omega_h)},
\end{equation}
where the last inequality follows from \eqref{eq:v_H3_stability}.

Applying integration by parts to the left side $-\int_\Omega \Delta e_h e_h$, we obtain
\begin{align}
\| \nabla e_h \|^2_{L^2(\Omega)}
\leq 
&\underbrace{\left| (r, \Delta v_h)_{L^2(\Omega_h)} \right| }_{\RNum{1}}
+ \underbrace{ \left| (\Delta e_h, \Delta v)_{L^2({\Omega_h \setminus \Omega})} \right| }_{\RNum{2}} \notag \\
& + \underbrace{ \left| (\Delta e_h, \Delta (v - v_h))_{L^2(\Omega_h)} \right| }_{\RNum{3}}
+ \underbrace{ {\big| j_h(\TCI_h \Tu, v_h) \big| + \big| j_h(\xi_h, v_h) \big| }
}_{\RNum{4}} 
+ \underbrace{ \big| \int_\Gamma e_h \partial_n e_h \big| }_{\RNum{5}} . \notag
\end{align}
In what follows, we systematically estimate the terms $\RNum{1}$ through $\RNum{5}$.

As a first step, we establish an upper bound for the error on the boundary, denoted by $\|e_h\|_{L^2(\Gamma)}$. 
Note that the exact solution extension satisfies $\Tu = 0$ precisely on the continuous boundary $\Gamma$. Thus, the error on the interface simplifies to $e_h = \Tu - u_h = -u_h$. By introducing the auxiliary variables $w = \Tu/\phi$ and $w_h = u_h/\phi_h$, we can rewrite the error on $\Gamma$ as $e_h = (\phi - \phi_h)w_h$. Splitting this expression yields:
\begin{align*}
\|e_h \|_{L^2(\Gamma)} 
&= \| (\phi_h - \phi) \big[ (w_h - w) + w \big] \|_{L^2(\Gamma)} \\
&\le \| \phi_h - \phi \|_{L^2(\Gamma)} \| w \|_{L^\infty(\Gamma)} 
+ \| \phi_h - \phi \|_{L^\infty(\Gamma)} \| w_h - w \|_{L^2(\Gamma)} .
\end{align*}
By virtue of the Hardy inequality and the Sobolev embedding theorem for dimension $d \in \{2,3\}$, the $L^\infty$-norm of $w$ is uniformly bounded by the regularity of $\Tu$:
\begin{equation} \label{eq:w_Linf_bound}
\| w \|_{L^\infty(\Omega_h)} \le C \| \Tu \|_{W^{1,\infty}(\TOmega)} \le C \| \Tu \|_{H^{k+1}(\TOmega)}.
\end{equation}
Applying the scaled trace inequality, the $L^2(\Gamma)$-norms can be controlled by:
\begin{align}
\| w_h - w \|_{L^2(\Gamma)} &\le C \left( h^{-1/2} \| w_h - w \|_{L^2(\Omega_h)} + h^{1/2} \| w_h - w \|_{H^1(\Omega_h)} \right), \label{eq:trace_wh} \\
\| \phi_h - \phi \|_{L^2(\Gamma)} &\le C \left( h^{-1/2} \| \phi_h - \phi \|_{L^2(\Omega_h)} + h^{1/2} \| \phi_h - \phi \|_{H^1(\Omega_h)} \right) \le C h^{k+1/2}. \label{eq:trace_phi}
\end{align}
To estimate the geometric error in the $L^\infty(\Gamma)$-norm, combining the inverse inequality with standard Sobolev embeddings for $d \le 3$ yields:
\begin{align}
\|\phi_h - \phi\|_{L^\infty(\Gamma)} 
&\le \|\phi_h - \CI_h \phi \|_{L^\infty(\Omega_h)} + \|\phi - \CI_h \phi \|_{L^\infty(\Omega_h)} \notag \\
&\le C h^{-1/2} \| \phi_h - \CI_h \phi \|_{L^6(\Omega_h)} + C h^{k-1/2} \notag \\
&\le C h^{-1/2} \| \phi_h - \CI_h \phi \|_{H^1(\Omega_h)} + C h^{k-1/2} 
\le C h^{k-1/2}. \label{eq:phi_Linf}
\end{align}
Applying the Hardy inequality on $\Omega_h$ to $\phi_h(w_h - w)$ and substituting the identity $\phi_h(w_h - w) = u_h - \phi_h w = -e_h + (\phi - \phi_h)w$, we deduce:
\begin{align}
\| w_h - w \|_{L^2(\Omega_h)} &\le C \| \phi_h(w_h - w) \|_{H^1(\Omega_h)}, \notag \\
&\le \| e_h \|_{H^1(\Omega_h)} + \| (\phi - \phi_h)w \|_{H^1(\Omega_h)} \notag \\
&\le \| e_h \|_{H^1(\Omega_h)} + \| \phi - \phi_h \|_{H^1(\Omega_h)} \| w \|_{L^\infty(\Omega_h)} + \| \phi - \phi_h \|_{L^\infty(\Omega_h)} \| w \|_{H^1(\Omega_h)} \notag \\
&\le \| e_h \|_{H^1(\Omega_h)} + C h^{k-1/2}\| \Tu \|_{H^{k+1}(\TOmega)} 
\label{eq:prod_H1} \\
\| w_h - w \|_{H^1(\Omega_h)} 
&\le C \| \phi_h(w_h - w) \|_{H^2(\Omega_h)} \notag \\
&\le \| e_h \|_{H^2(\Omega_h)} 
+ \| (\phi - \phi_h) w \|_{H^2(\Omega_h)} \notag \\
&\le \| e_h \|_{H^2(\Omega_h)} 
+ \| \phi - \phi_h\|_{H^2(\Omega_h)} 
\|w \|_{W^{1,\infty}(\Omega_h)} \notag\\
& \quad+\| \phi - \phi_h\|_{L^\infty(\Omega_h)} 
\|w \|_{H^2(\Omega_h)} \notag \\
&\le \| e_h \|_{H^2(\Omega_h)} 
+ C h^{k-1} 
\|\Tu \|_{H^{k+1}(\TOmega)} 
\label{eq:prod_H2}
\end{align}
where the last inequality uses the estimate from \eqref{H2-approx}.
Substituting these bounds back into the trace inequality \eqref{eq:trace_wh} yields:
\begin{align}
\| w_h - w \|_{L^2(\Gamma)} 
&\le C h^{-1/2} \Big( \| e_h \|_{H^1(\Omega_h)} + C h^{k-1/2}\| \Tu \|_{H^{k+1}(\TOmega)} \Big) \notag \\
&\quad + C h^{1/2} \Big( \| e_h \|_{H^2(\Omega_h)} 
+ C h^{k-1} \|\Tu \|_{H^{k+1}(\TOmega)} \Big)\notag
\\
&\le C \Big( h^{-1/2} \| e_h \|_{H^1(\Omega_h)} 
+ h^{1/2} \| e_h \|_{H^2(\Omega_h)}
+ h^{k-1}\| \Tu \|_{H^{k+1}(\TOmega)} \Big) 
\label{eq:w_h-w_Gamma}.
\end{align}
Consequently, returning to the initial splitting of $\|e_h\|_{L^2(\Gamma)}$ and expanding the terms with \eqref{eq:w_h-w_Gamma}, we obtain:
\begin{align}
\| e_h \|_{L^2(\Gamma)}
&\le \| \phi_h-\phi \|_{L^2(\Gamma)} \|w\|_{L^\infty(\Gamma)} 
+ \| \phi_h-\phi \|_{L^\infty(\Gamma)} \|w_h - w\|_{L^2(\Gamma)} \notag \\
&\le C h^{k+1/2}\|\Tu\|_{H^{k+1}(\TOmega)} \notag \\
&\quad + C h^{k-1/2} \Big( h^{-1/2} \| e_h \|_{H^1(\Omega_h)} + h^{1/2} \| e_h \|_{H^2(\Omega_h)} + h^{k-1}\| \Tu \|_{H^{k+1}(\TOmega)} \Big) \notag \\
&\le C h^{k+1/2}\|\Tu\|_{H^{k+1}(\TOmega)} + C h^{k} \| e_h \|_{H^2(\Omega_h)} + C h^{k-1} \| e_h \|_{H^1(\Omega_h)} . \label{eq:eh_Gamma_expanded}
\end{align}
For polynomial degree $k \ge 2$, the explicitly known high-order term $h^{2k-3/2}\|\Tu\|_{H^{k+1}}$ is asymptotically dominated by the leading-order error $h^{k+1/2}\|\Tu\|_{H^{k+1}}$ (since $2k-3/2 \ge k+1/2$) and can thus be absorbed. Furthermore, by invoking the previously established $H^2$-error bound \eqref{H2_Error}, the term involving the $H^2$-error can be controlled as $h^k \|e_h\|_{H^2(\Omega_h)} \le  C h^{2k-1} \|f\|_{H^{k-1}(\Omega_h)} $. Since $2k-1 \ge k+1/2$ for $k \ge 2$, this term is also bounded by the leading-order term and subsequently absorbed. Finally, applying the standard elliptic regularity $\|\Tu\|_{H^{k+1}(\TOmega)} \le C \|f\|_{H^{k-1}(\Omega_h)}$, the remaining unabsorbed terms yield the fundamental boundary estimate:
\begin{equation}
\| e_h \|_{L^2(\Gamma)} \le C h^{k+1/2}\|f\|_{H^{k-1}(\Omega_h)} + C h^{k-1} \| e_h \|_{H^1(\Omega_h)}.
\label{eq:eh_on_Gamma}
\end{equation}

By \cite[Lemma 3.1]{Lozinski2019}, 
for any sufficiently smooth function $\zeta \in H^1(\Omega_h \setminus \Omega)$,
\begin{equation}\label{eq:strip_taylor}
\| \zeta \|_{L^2(\Omega_h \setminus \Omega)} \le C \Big( h^{1/2} \| \zeta \|_{L^2(\Gamma)} + h \| \nabla \zeta \|_{L^2(\Omega_h \setminus \Omega)} \Big).
\end{equation}
Applying \eqref{eq:strip_taylor} to $\zeta = \Delta v$, we note that the dual problem \eqref{eq:dual_problem} ensures $-\Delta v = -e_h$ up to the continuous boundary $\Gamma$, yielding the exact trace identity $\Delta v = e_h$ on $\Gamma$. Furthermore, the gradient of the Laplacian is controlled by the $H^3$-regularity of the dual extension, such that $\| \nabla (\Delta v) \|_{L^2(\Omega_h \setminus \Omega)} \le C\| v \|_{H^3(\TOmega)} 
\le C\| e_h \|_{H^1(\Omega_h)}$. Consequently, we obtain
\begin{align}
\| \Delta v \|_{L^2(\Omega_h \setminus \Omega)} 
&\le C \Big( h^{1/2} \| \Delta v \|_{L^2(\Gamma)} + h \| \nabla (\Delta v) \|_{L^2(\Omega_h \setminus \Omega)} \Big) \notag \\
&\le C h^{1/2} \| e_h \|_{L^2(\Gamma)} + C h \| v \|_{H^3(\TOmega)} \notag \\
&\le C h^{k+1} \|f\|_{H^{k-1}(\Omega_h)} + C h \| e_h \|_{H^1(\Omega_h)}, \label{eq:Delta_v_strip}
\end{align}
where, assuming $k \ge 2$, the higher-order term $h^{k-1/2} \| e_h \|_{H^1(\Omega_h)}$ is asymptotically dominated and hence safely absorbed by the leading-order term $h \| e_h \|_{H^1(\Omega_h)}$.
Utilizing this estimate, the term $\RNum{2}$ can be bounded as follows:
\begin{align}
\RNum{2}
&\le \| e_h \|_{H^2(\Omega_h)} \Big( C h^{k+1} \|f\|_{H^{k-1}(\Omega_h)} + C h \| e_h \|_{H^1(\Omega_h)} \Big) \notag \\
&\le C h^{2k} \| f \|^2_{H^{k-1}(\Omega_h)} + C h^k \| f \|_{H^{k-1}(\Omega_h)} \| e_h \|_{H^1(\Omega_h)} . \label{eq:bound_H1_T2}
\end{align}

For the first term $\RNum{1}$, we bound the $L^2$-norm of $\Delta v_h$ on the narrow band $\Omega_h \setminus \Omega$ by applying the triangle inequality and the approximation property of the 
PHT-spline interpolation:
% Lagrange interpolation:
\begin{align*}
\| \Delta v_h \|_{L^2(\Omega_h \setminus \Omega)} 
&\le \| \Delta v \|_{L^2(\Omega_h \setminus \Omega)} + \| \Delta(v_h - v) \|_{L^2(\Omega_h \setminus \Omega)} \\
&\le \| \Delta v \|_{L^2(\Omega_h \setminus \Omega)} 
+ C h \| v \|_{H^3(\TOmega)} \\
&\le C h^{k+1} \|f\|_{H^{k-1}(\Omega_h)} + C h \| e_h \|_{H^1(\Omega_h)} .
\end{align*}
Noting that the residual $r$ vanishes identically in the physical domain $\Omega$, a direct application of \eqref{equ:r_f} alongside the Cauchy-Schwarz inequality yields the bound for $\RNum{1}$:
\begin{align}
\RNum{1}
= \left| (r, \Delta v_h)_{L^2(\Omega_h)} \right|
&\le \| r \|_{L^2(\Omega_h \setminus \Omega)} \| \Delta v_h \|_{L^2(\Omega_h \setminus \Omega)} \notag \\
&\le C h^{2k} \| f \|^2_{H^{k-1}(\Omega_h )} + C h^k \| f \|_{H^{k-1}(\Omega_h )} \| e_h \|_{H^1(\Omega_h)}. \label{eq:bound_H1_T1}
\end{align}

By virtue of the interpolation estimate in Lemma~\ref{lem:interpolation} and the established $H^2$-error bound \eqref{H2_Error}, the third term $\RNum{3}$ is controlled as follows:
\begin{align}
\RNum{3}
&= \left| (\Delta e_h, \Delta (v - v_h))_{L^2(\Omega_h)} \right| \notag \\
&\leq C h \|v\|_{H^3(\TOmega)}
\| e_h \|_{H^2(\Omega_h)}  
\leq C h^k \| f \|_{H^{k-1}(\Omega_h)}
\| e_h \|_{H^1(\Omega_h)}. \label{eq:bound_H1_T3}
\end{align}

To estimate the fourth term $\RNum{4}$, we decompose the jump penalty on the interface faces $F \in \mathcal{F}_h^\Gamma$ into high-order ($3 \le j \le 2k$) and low-order ($j = 2$) contributions. For the high-order jumps, successive applications of the discrete scaled trace and inverse inequalities give:
\begin{align}
\| \jump{\partial_n^j v_h} \|_{L^2(F)}
&\leq C h^{-1/2} \| \partial_n^j v_h \|_{L^2(K)} \notag \\
&\leq C h^{5/2 - j} \| \partial_n^3 v_h \|_{L^2(K)} 
\leq C h^{5/2 - j} \| v_h \|_{H^3(K)}. 
\label{eq:jump_trace_inverse_mid}
\end{align}
Summing this estimate over all relevant faces for $j \ge 3$, we deduce:
\begin{align*}
\sum_{j=3}^{2k} \sum_{F\in \mathcal{F}_h^\Gamma} h^{2j-3}\| \jump{\partial_n^j v_h} \|^2_{L^2(F)}
\leq C \sum_{j=3}^{2k} \sum_{K\in \CT_h^\Gamma} h^{2j-3} h^{5 - 2j}\| v_h\|^2_{H^3(K)} 
\leq C h^2 \| v_h\|^2_{H^3_h(\Omega_h)} .
\end{align*}
For the low-order jump $j = 2$, the assumed regularity $v \in H^3(\TOmega)$ ensures that the exact normal derivative is continuous across the faces, i.e., $\jump[F]{\partial^2_n v} = 0$. Applying the continuous scaled trace inequality to the error function $\partial_n^2 (v_h - v) \in H^1(K)$ in conjunction with the interpolation property \eqref{interpolation_Hj}, we obtain:
\begin{align*}
\sum_{F\in \mathcal{F}_h^\Gamma} h \| \jump{\partial_n^2 v_h} \|^2_{L^2(F)}
&= \sum_{F\in \mathcal{F}_h^\Gamma} h \| \jump{\partial_n^2 (v_h-v)} \|^2_{L^2(F)} \\
&\leq C \sum_{K\in \CT_h^\Gamma} \Big( \| \partial_n^2 (v_h-v) \|^2_{L^2(K)} + h^2 \| \nabla \partial_n^2 (v_h-v) \|^2_{L^2(K)} \Big) \notag \\
&\leq C h^2 \| v \|^2_{H^{3}(\TOmega)}.
\end{align*}
Combining these estimates for both low- and high-order jump contributions, we arrive at the unified bound for the stabilization penalty associated with the dual interpolant $v_h$:
\begin{align}\label{eq:penalty_vh_bound}
j_h(v_h,v_h)
&= \sum_{j=2}^{2k} \sum_{F\in \mathcal{F}_h^\Gamma} h_F^{2j-3} \| \jump{\partial_n^j v_h} \|^2_{L^2(F)} \notag \\ 
&\le C h^2 \| v \|^2_{H^3(\TOmega)}
+ C h^2 \| v_h \|^2_{H^3_h(\Omega_h)} 
\le C h^2 \| e_h \|^2_{H^1(\Omega_h)}.
\end{align}

To bound the fourth term $\RNum{4}$, we apply the Cauchy-Schwarz inequality to the stabilization bilinear form. Invoking the jump penalty estimate $j_h^{1/2}(\TCI_h\Tu, \TCI_h\Tu)$ from \eqref{eq:penalty_Ihu} alongside the energy norm bound for $\xi_h$ from \eqref{Energy_norm_xi}, we deduce:
\begin{align}
\RNum{4} 
\leq \Big( j_h^{1/2}(\TCI_h\Tu, \TCI_h\Tu)
+ j_h^{1/2}(\xi_h, \xi_h) \Big)
j_h^{1/2}(v_h, v_h) 
\le C h^k \| f \|_{H^{k-1}(\Omega_h)} \| e_h \|_{H^1(\Omega_h)}. \label{eq:bound_H1_T4}
\end{align}

For the boundary integral term $\RNum{5}$, an application of the Cauchy--Schwarz inequality on the interface $\Gamma$ yields:
\begin{equation} \label{eq:R5_CS}
\RNum{5} = \left| \int_\Gamma e_h \partial_n e_h \, ds \right| \le \| e_h \|_{L^2(\Gamma)} \| \partial_n e_h \|_{L^2(\Gamma)}.
\end{equation}
The boundary trace of the normal derivative is bounded by combining the standard scaled trace inequality with the prior $H^2$-error estimate \eqref{H2_Error}:
\begin{align}
\| \partial_n e_h \|_{L^2(\Gamma)} 
&\le C h^{-1/2} \| e_h \|_{H^1(\Omega_h)} + C h^{1/2} \| e_h \|_{H^2(\Omega_h)} \notag \\
&\le C h^{-1/2} \| e_h \|_{H^1(\Omega_h)} + C h^{k-1/2} \| f \|_{H^{k-1}(\Omega_h)}. \label{eq:dn_eh_trace}
\end{align}
Substituting the previously derived bound for $\| e_h \|_{L^2(\Gamma)}$ from \eqref{eq:eh_on_Gamma} and the normal derivative estimate \eqref{eq:dn_eh_trace} into \eqref{eq:R5_CS}, we obtain:
\begin{align}
\RNum{5} 
&\le \Big( C h^{k+1/2}\|f\|_{H^{k-1}(\Omega_h)} + C h^{k-1} \| e_h \|_{H^1(\Omega_h)} \Big) \notag \\
&\quad \times \Big( C h^{-1/2} \| e_h \|_{H^1(\Omega_h)} + C h^{k-1/2} \| f \|_{H^{k-1}(\Omega_h)} \Big) \notag \\
&\le C h^k \| f\|_{H^{k-1}(\Omega_h)} \|e_h\|_{H^1(\Omega_h)} + C h^{2k}\|f\|^2_{H^{k-1}(\Omega_h)} + C h^{k-3/2} \| e_h \|^2_{H^1(\Omega_h)}. \label{eq:bound_H1_T5}
\end{align}

Finally, collecting the individual estimates \eqref{eq:bound_H1_T2}, \eqref{eq:bound_H1_T1}, \eqref{eq:bound_H1_T3}, \eqref{eq:bound_H1_T4}, and \eqref{eq:bound_H1_T5}, we arrive at the following comprehensive bound for the gradient error over the physical domain $\Omega$:
\begin{align}
\| \nabla e_h \|^2_{L^2(\Omega)}
&\leq \RNum{1} + \RNum{2} + \RNum{3} + \RNum{4} + \RNum{5} \notag \\
&\lesssim h^k \|f\|_{H^{k-1}(\Omega_h)}
\| e_h \|_{H^1(\Omega_h)}
+ h^{2k} \|f\|^2_{H^{k-1}(\Omega_h)}
+ h^{k-3/2}\| e_h \|^2_{H^1(\Omega_h)}. \label{e_h_semi}
\end{align}

To extend the optimal $H^1$-estimate to the entire computational domain $\Omega_h$, we leverage the discrete extension property of the ghost penalty operator \cite{Burman2010} in conjunction with a generalized Poincar\'e-Friedrichs inequality on the physical domain $\Omega$. We then substitute the gradient bound \eqref{e_h_semi}, the ghost penalty energy estimate \eqref{Energy_norm_xi}, and the interfacial $L^2$-norm estimate \eqref{eq:eh_on_Gamma}. By further bounding the interpolation trace error via the standard continuous trace inequality $\| \rho \|_{L^2(\Gamma)} \le C \| \rho \|_{H^1(\Omega)} \le C h^k \| f \|_{H^{k-1}(\Omega_h)}$, we uniformly estimate the discrete component $\xi_h \in S_h$ using the triangle inequality and Corollary \ref{cor:H1_extension} to obtain: 
\begin{align}
\|\xi_h\|^2_{H^1(\Omega_h)}
&\le C \Big( \|\nabla \xi_h\|^2_{L^2(\Omega)} + \|\xi_h\|^2_{L^2(\Gamma)} + h^2 j_h(\xi_h, \xi_h) \Big) \notag \\
&\le C \Big( \|\nabla e_h\|^2_{L^2(\Omega)} + \|e_h\|^2_{L^2(\Gamma)} + h^2 \TN{\xi_h}^2 + \|\rho\|_{H^1(\Omega_h)}^2 + \|\rho\|_{L^2(\Gamma)}^2 \Big) \notag \\
&\le C \Big( h^{k} \|e_h\|_{H^1(\Omega_h)}  \|f\|_{H^{k-1}(\Omega_h)} + h^{2k}\|f\|^2_{H^{k-1}(\Omega_h)} + h^{k-3/2}\|e_h\|^2_{H^1(\Omega_h)} \Big), 
\notag
%\label{eq:xi_h_H1_bound}
\end{align}
where the sub-optimal error term $h^{2k-2}\|e_h\|^2_{H^1(\Omega_h)}$ is asymptotically dominated and thus safely absorbed by the leading-order component $h^{k-3/2}\|e_h\|^2_{H^1(\Omega_h)}$ for any polynomial degree $k \ge 2$. Applying Young's inequality with an arbitrarily small parameter $\varepsilon > 0$ to decouple the cross term, and subsequently reverting to the total error $e_h = \xi_h + \rho$, we deduce:
\begin{align*}
\|e_h\|_{H^1(\Omega_h)}^2 
&\le 2\|\xi_h\|_{H^1(\Omega_h)}^2 + 2\|\rho\|_{H^1(\Omega_h)}^2 \\
&\le C_1 \varepsilon^{-1} h^{2k}\|f\|^2_{H^{k-1}(\Omega_h)} + C_2 (\varepsilon + h^{k-3/2}) \|e_h\|_{H^1(\Omega_h)}^2.
\end{align*}
Because the polynomial degree $k \ge 2$ guarantees a strictly positive exponent $k - 3/2$, we can select a sufficiently small $\varepsilon$ and assume an adequately refined mesh size $h$ to ensure that $C_2 (\varepsilon + h^{k-3/2}) < 1$. This permits the rigorous absorption of the remaining $H^1$-error terms into the left-hand side, thereby completing the proof of the optimal global estimate $\| e_h \|_{H^1(\Omega_h)} \le C h^k \|f\|_{H^{k-1}(\Omega_h)}$.
\hfill\end{proof}

\section{\emph{A posteriori} error analysis}
\label{sec:poster}

\subsection{Reliability and efficiency of the error indicator}
\label{sec4.1}
A fundamental advantage of the PHT-spline space is its global $C^1$-continuity. This higher-order smoothness circumvents the need for interface jump residuals---which typically complicate the analysis of standard $C^0$ finite element methods---and enables the formulation of an \emph{a posteriori} error indicator based exclusively on the interior $L^2$-residual. 

\begin{proposition}
[Reliability and efficiency of the \emph{a posteriori} indicator]
\label{prop:poster_error_estimate}
Let $u_h\in S_h$ denote the discrete solution, and let $\Tu\in H^{k+1}(\TOmega)$ be the exact extension of the true solution $u$, satisfying $-\Delta \Tu = f + r$ in $\Omega_h$ with the perturbation $r \equiv 0$ in the physical domain $\Omega$. We define the local and global \emph{a posteriori} error indicators by
\begin{equation}
\eta_T := \| \Delta u_h + f \|_{L^2(T)} \quad \forall\, T\in \CT_h, \qquad
\eta := \left( \sum_{T \in \CT_h} \eta_T^2 \right)^{1/2},
\end{equation}
respectively. Then, the error $e_h := \Tu - u_h$ satisfies the reliability bound
\begin{equation}\label{poster_reliable}
\| \Delta e_h\|_{L^2(\Omega_h)} \leq \eta + \| r\|_{L^2(\Omega_h)},
\end{equation}
and the efficiency bound
\begin{equation}\label{poster_efficient}
\eta \leq \| \Delta e_h\|_{L^2(\Omega_h)} + \| r \|_{L^2(\Omega_h)}.
\end{equation}
\end{proposition}

\begin{proof}
Recalling the strong form $f = -\Delta \Tu - r$ in $\Omega_h$ and substituting this relation into the definition of the global indicator $\eta$, we obtain
\begin{align*}
\eta 
&= \| \Delta u_h + f \|_{L^2(\Omega_h)} \\
&= \| \Delta u_h - \Delta \Tu - r \|_{L^2(\Omega_h)} \\
&= \| -\Delta e_h - r \|_{L^2(\Omega_h)} 
= \| \Delta e_h + r \|_{L^2(\Omega_h)} .
\end{align*}
Applying the triangle inequality directly yields the reliability upper bound \eqref{poster_reliable}:
\begin{equation*}
\| \Delta e_h \|_{L^2(\Omega_h)} 
= \| (\Delta e_h + r) - r \|_{L^2(\Omega_h)} \le \| \Delta e_h + r \|_{L^2(\Omega_h)} + \| r \|_{L^2(\Omega_h)} = \eta + \| r \|_{L^2(\Omega_h)}.
\end{equation*}
A completely analogous application of the triangle inequality to $\eta = \| \Delta e_h + r \|_{L^2(\Omega_h)}$ establishes the efficiency lower bound \eqref{poster_efficient}:
\begin{equation*}
\eta \le \| \Delta e_h \|_{L^2(\Omega_h)} + \| r \|_{L^2(\Omega_h)},
\end{equation*}
which completes the proof.
\hfill\end{proof}

\begin{remark}
Proposition~\ref{prop:poster_error_estimate} shows that the indicator
$\eta$ provides two-sided bounds for the Laplacian error up to the
perturbation term $\|r\|_{L^2(\Omega_h)}$. Indeed,
\[
\left|
\eta-\|\Delta e_h\|_{L^2(\Omega_h)}
\right|
\le
\|r\|_{L^2(\Omega_h)}.
\]
Using
$\|\Delta e_h\|_{L^2(\Omega_h)}
\le
\sqrt d\,|e_h|_{H^2(\Omega_h)}$
together with \eqref{H2_Error} and \eqref{equ:r_f}, we obtain
\[
\eta
\le
\|\Delta e_h\|_{L^2(\Omega_h)}
+
\|r\|_{L^2(\Omega_h)}
\le
C h^{k-1}.
\]
Thus, the indicator $\eta$ decays with the optimal-order upper bound
$\mathcal O(h^{k-1})$.  
\end{remark}

\subsection{Adaptive algorithm}
\label{sec4.2}
Based on the \emph{a posteriori} error indicator defined above, we implement a standard adaptive iteration loop of the form \textbf{SOLVE} $\to$ \textbf{ESTIMATE} $\to$ \textbf{MARK} $\to$ \textbf{REFINE} \cite{Dorfler1996,Ainsworth1997,Chen2021} to efficiently capture local features of the solution on the hierarchical T-mesh. The detailed procedure is outlined in Algorithm \ref{algo}.

Starting from an initial uniform background mesh $\HCT_h^{(0)}$, the algorithm iterates until the error tolerance $\mathrm{TOL}$ is met. In the \textbf{MARK} step, we employ the standard Dörfler marking strategy with a parameter $\theta \in (0,1)$ to select a minimal subset of elements $\mathcal{M}^{(\ell)}$ that contribute significantly to the global error $\eta^{(\ell)}$. In the \textbf{REFINE} step, the marked elements undergo a one-to-four (or one-to-eight in 3D) spatial subdivision. Furthermore, recursive refinement is triggered on adjacent cells to strictly enforce the 1-regularity constraint. Finally, the active computational mesh $\CT_h^{(\ell+1)}$ is dynamically identified 
by selecting the elements intersecting the discrete domain 
$\{\phi^{(\ell+1)}_h < 0\}$.

\begin{algorithm}[t]
\caption{Adaptive refinement algorithm}
\label{algo}
\begin{algorithmic}[1]
\Require Error tolerance $\mathrm{TOL}>0$, marking parameter $\theta\in(0,1)$, and maximum iteration limit $\ell_{\rm max} > 0$.

\State \textbf{Initialize:} Construct an initial uniform background mesh $\hat{\mathcal T}_h^{(0)}$ with $N_0\times N_0$ cells. Set $\ell \gets 0$ and $\eta^{(0)} \gets +\infty$.
\While{$\eta^{(\ell)} > \mathrm{TOL}$ and $\ell \le \ell_{\rm max}$}
    \State \textbf{Step 1:} From the background mesh $\hat{\mathcal T}_h^{(\ell)}$ construct the
    active mesh $\mathcal T_h^{(\ell)}\subset \hat{\mathcal T}_h^{(\ell)}$ based on the current level-set function $\phi^{(\ell)}_h$.
    \State \textbf{Step 2:} Assemble the discrete system on $\mathcal T_h^{(\ell)}$ and solve the resulting linear system to obtain $u_h^{(\ell)}$.
    \State \textbf{Step 3:} Compute the local indicators $\eta^{(\ell)}_T = \|\Delta u_h^{(\ell)} + f \|_{L^2(T)} $
    for $T \in \CT_h^{(\ell)}$,
    and the global indicator
    \(
        \eta^{(\ell)} := \Big(\sum_{T\in \CT_h^{(\ell)}} ( \eta_T^{(\ell)})^2 \Big)^{1/2}.
    \)
    \State \textbf{Step 4:} Find a minimal set of marked cells
    $\mathcal M^{(\ell)}\subset \CT_h^{(\ell)}$ such that
    \[
        \sum_{T\in \mathcal M^{(\ell)}} ( \eta_T^{(\ell)})^2 \;\ge\; \theta \sum_{T\in \CT_h^{(\ell)}} ( \eta_T^{(\ell)})^2
        \;=\; \theta \big(\eta^{(\ell)}\big)^2.
    \]
    \State \textbf{Step 5:} Apply one-to-four (one-to-eight in 3D) subdivision to each marked cell $T \in \mathcal M^{(\ell)}$ and recursively refine the necessary neighboring cells to strictly enforce the 1-regular constraint. This process yields the valid hierarchical refined mesh $\HCT_h^{(\ell+1)}$. 
    \State \textbf{Update:} Set $\ell \gets \ell+1$.
\EndWhile
\end{algorithmic}
\end{algorithm}

\section{Numerical experiments}
\label{sec:test}
All numerical experiments are implemented using the MATLAB-based package \texttt{IGAPack} \cite{Anitescu2018}, employing the hierarchical T-mesh described in the previous sections. Unless otherwise specified, we utilize bicubic PHT-splines (polynomial degree $k=3$) for the discrete space $ V_h$. In the following tests, we take tolerance $\mathrm{TOL} = 10^{-8}$, marking parameter $\theta = 0.25$ and maximum iteration $\ell_{\rm max} = 20$ in Algorithm~\ref{algo}.

\subsection{Example 1}
\label{example1}
We consider a smooth level-set function,
\begin{align}\label{example1_phi}
\phi(x,y) = (x - 0.5)^2 
+ (y - 0.5)^2 - r^2,
\qquad (x,y) \in (0,1)^2,
\end{align}
with $r=0.3$.
The exact solution is chosen as
$u_{\text{ex}}(x, y) 
= \phi(x,y) \exp(x+y)$,
and the corresponding right-hand side is computed analytically.
Here $u_{\rm ex}$ is explicitly defined on the entire background domain $\hat\Omega$, so the errors evaluated on $\Omega_h$ involve this fixed extension and have no extension ambiguity. 

We consider uniformly refined meshes, starting from an initial coarse mesh consisting of $16 \times 16$ cells. These meshes do not contain hanging nodes and satisfy $h_M / h_m = 1$. The accuracy is evaluated by computing the errors $\| u_h - u_{\text{ex}} \|_{H^1(\Omega_h)}$ and $\| u_h - u_{\text{ex}} \|_{H^2(\Omega_h)}$ on the background mesh. As reported in Table \ref{Tab_Uniform_Error}, the proposed method achieves the optimal convergence rates predicted by the theory: $\mathcal{O}(h^3)$ in the $H^1$-norm and $\mathcal{O}(h^2)$ in the $H^2$-norm.

\begin{table}[htbp]
\centering
\caption{Convergence results with uniformly refined meshes for Example~\ref{example1}.}
\label{Tab_Uniform_Error}
% \vspace{-15pt}
\begin{tabular}{lcccc}
\toprule
$h$ & $\| u_h - u_{\text{ex}}\|_{H^1(\Omega_h)}$ & order 
& $\| u_h - u_{\text{ex}} \|_{H^2(\Omega_h)}$ & order \\
\midrule
1/16  &1.58E-07	& --	&1.63E-05	& --  \\
1/32 &1.94E-08	&3.03	&4.00E-06	&2.02 \\
1/64 &2.39E-09	&3.02	&9.90E-07	&2.02 \\
1/128&2.98E-10	&3.00	&2.47E-07	&2.00 \\
\bottomrule
\end{tabular}
% \vspace{-0.5cm}
\end{table}

\subsection{Example 2}
\label{example2}
Consider the shifted coordinates
$X := x-\tfrac12$, $Y := y-\tfrac12$ with the level-set function:
\begin{equation}
\phi(x,y)
:= X^2+Y^2-r_0^2
\;-\;A\sum_{j=0}^{5}\exp\!\left(-\frac{(X-X_j)^2+(Y-Y_j)^2}{\sigma_0^{\,2}}\right),
\end{equation}
where the bump centers $(X_j,Y_j)$ are distributed on a circle of radius $r_0$,
\begin{equation}\label{eq:bump_centers_2d}
(X_j,Y_j) := \bigl(r_0\cos\theta_j,\; r_0\sin\theta_j\bigr),
\qquad \theta_j:=\frac{2\pi j}{6},\qquad j=0,\dots,5 .
\end{equation}
Let
$u_{\rm base}(x,y):=\exp\!\left(-\frac{(x-0.3)^2+(y-0.7)^2}{\sigma^{2}}\right)$,
we prescribe the exact solution
$u_{\rm ex}(x,y):=\phi(x,y)\,u_{\rm base}.$
In this test, we take $r_0=0.2$, $\sigma_0=0.1$, $A=0.5$ and $\sigma=0.06$.

Starting from an initial $32 \times 32$ background mesh, we compare the performance of uniform refinement against the proposed adaptive strategy. Fig.~\ref{Fig_Poster_Error} illustrates the evolution of the indicator $\eta$ and the exact errors with respect to the active DOFs. While both strategies successfully achieve the optimal asymptotic rates ($\mathcal{O}(h^3)$ in $H^1$-norm and $\mathcal{O}(h^2)$ in $H^2$-norm), the adaptive method drastically reduces the computational cost.

The superior efficiency is further quantified in Table~\ref{Tab_uniform-vs-adaptive}. The adaptive approach attains an accuracy level comparable to that of uniform refinement using nearly 60\% fewer DOFs (6128 vs. 14764). Conversely, for a fixed DOF budget, adaptivity reduces the numerical error by almost an order of magnitude. This targeted efficiency is visually confirmed in Fig.~\ref{Figure_Example2_Tmesh}; the generated adaptive T-mesh (3328 leaf cells) automatically concentrates the refinement in the local region $[0.2, 0.5] \times [0.5, 0.8]$ to strictly resolve the sharp peak at $(0.3, 0.7)$, demonstrating the robustness of the \emph{a posteriori} indicator.

\begin{table}[t]
\centering
\caption{Comparison between different refinements for Example \ref{example2}.}
\label{Tab_uniform-vs-adaptive}
\begin{tabular}{lrcc}
\hline
Refinement strategy & DOFs & $\|e_h\|_{H^1(\Omega_h)}$ & $\|e_h\|_{H^2(\Omega_h)}$ \\
\hline
Uniform   & 14764 & 4.82E-5 & 3.50E-2 \\
Adaptive  & 6128  & 3.03E-5 & 2.05E-2 \\
Adaptive  & 14932 & 4.82E-6 & 5.71E-3 \\
\hline
\end{tabular}
% \vspace{-0.5cm}
\end{table}

\begin{figure}[htbp]
	\centering
	\includegraphics[width=0.6\textwidth]{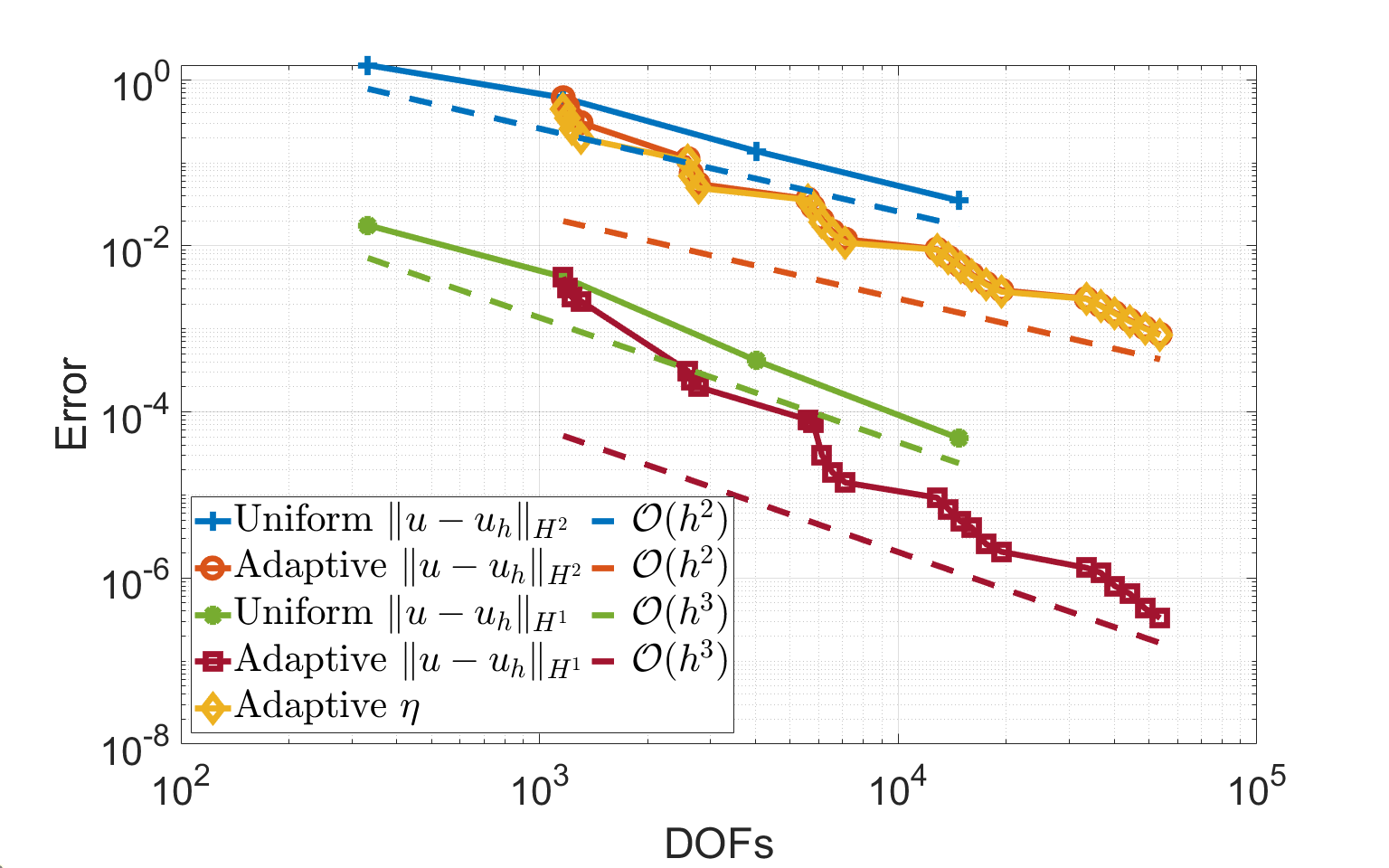}
    % \vspace{-10pt}
	\caption{Error versus DOFs on uniform and adaptively refined meshes for Example~\ref{example2}.}
    \label{Fig_Poster_Error}
    % \vspace{-1cm}
\end{figure}

\begin{figure}[htbp]
    \centering
    \includegraphics[width=0.35\textwidth]{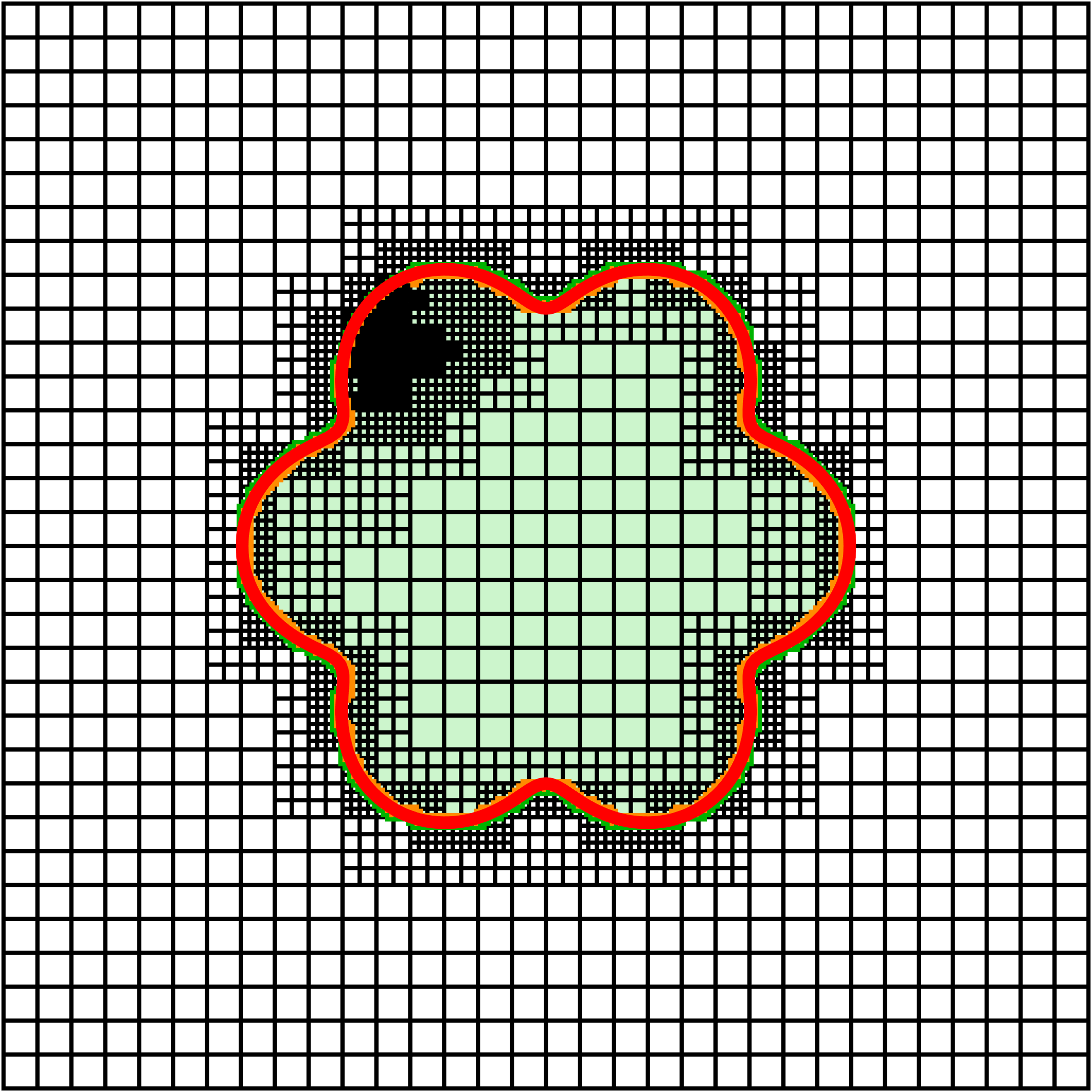}
    \qquad
    \includegraphics[width=0.35\textwidth]{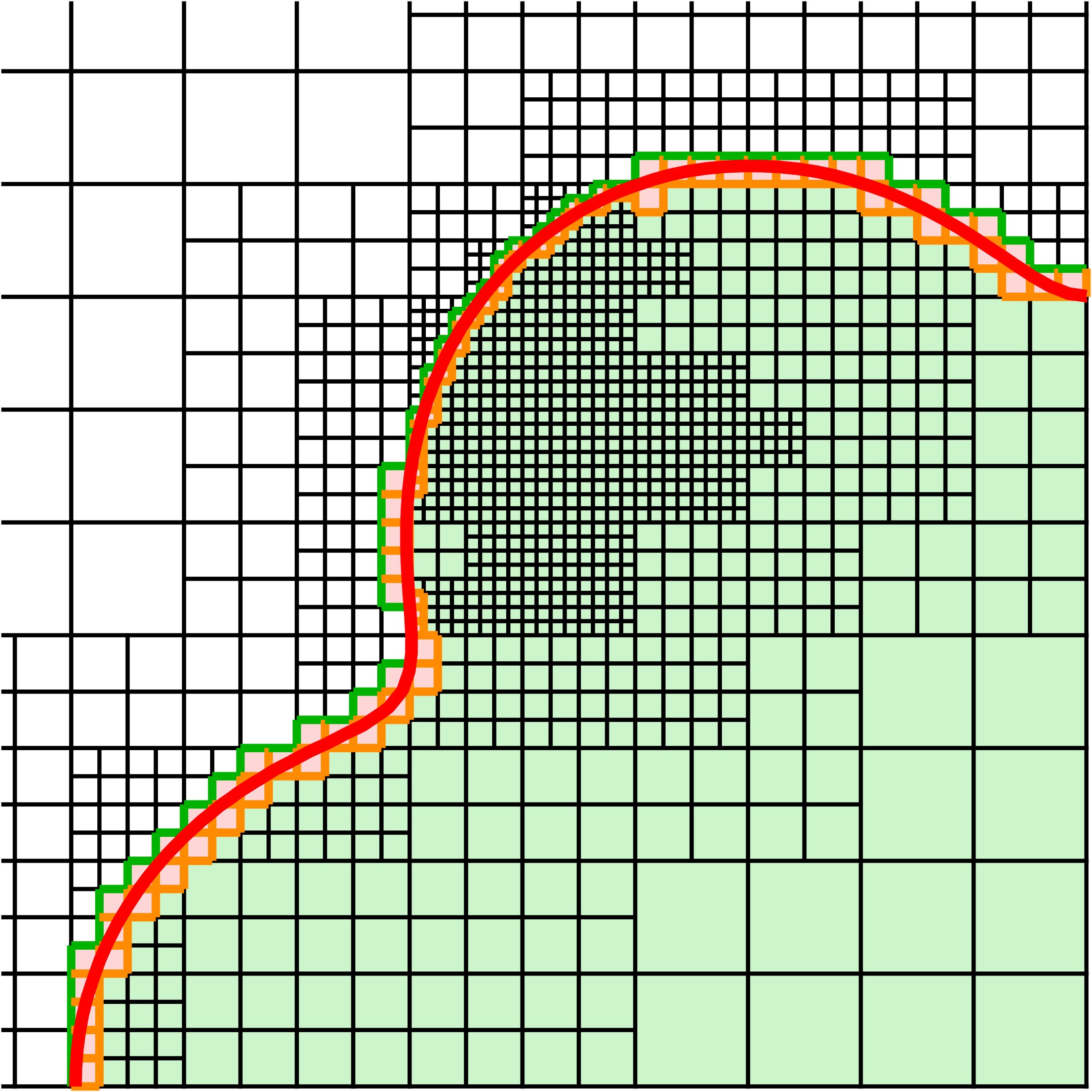}
    \caption{The computational mesh and its zoomed-in view of the region $[0.2,0.5]\times [0.5,0.8]$ with 3328 leaf cells. }
    \label{Figure_Example2_Tmesh}
    % \vspace{-1cm}
\end{figure}

\subsection{Example 3}
\label{example3}
In the last test, we move on to the three-dimensional case. Following \cite{Burman2015, phifem_heat}, we denote the shifted coordinates
$X:=x-\frac12$, 
$Y:=y-\frac12$, 
$Z:=z-\frac12$,
and the parameters
$r_0=0.25, \sigma_0=0.15, A=1, c=\frac{1}{\sigma_0^2}.$
Define the 12 bump centers on the sphere of radius $r_0$ as follows.
Let $s:= r_0 / \sqrt{5}$. For $j=0,1,2,3,4$,
\[
(X_j,Y_j,Z_j)=\bigl(2s\cos\theta_j,\;2s\sin\theta_j,\;s\bigr),\qquad
\theta_j=\frac{2j\pi}{5}.
\]
For $j=5,6,7,8,9$,
\[
(X_j,Y_j,Z_j)=\bigl(2s\cos\theta_j,\;2s\sin\theta_j,\;-s\bigr),\qquad
\theta_j=\frac{(2(j-5)-1)\pi}{5}.
\]
Finally, the two poles are
\[
(X_{10},Y_{10},Z_{10})=(0,0,r_0),\qquad (X_{11},Y_{11},Z_{11})=(0,0,-r_0).
\]
The level-set function $\phi$ is defined by
\begin{equation*}
\phi = X^2+Y^2+Z^2-r_0^2
-\sum_{j=0}^{11} A\exp\!\Big(-c\big((X - X_j)^2 + (Y-Y_j)^2 + (Z-Z_j)^2\big)\Big).
\end{equation*}
Next, define
$u_{\rm base}(x,y,z)=\sin(\pi x)\sin(\pi y)\sin(\pi z)$,
and the exact solution is taken as
$u_{\rm ex}(x,y,z)=\phi(x,y,z)u_{\rm base}(x,y,z)$.

Figure~\ref{Figure_Example3_Mesh} illustrates the zero-contour surface $\Gamma$ of the bumpy sphere, alongside a representative cutaway of the adaptively refined 3D hierarchical T-mesh. 
We evaluate the accuracy of the proposed method by tracking the errors on both uniformly and adaptively refined meshes. As depicted in Fig.~\ref{Figure_Example3}, the computed $H^1$ and $H^2$ errors retain their optimal asymptotic convergence rates. 
These 3D results confirm two critical properties of the proposed UnCut FEM framework: first, the PHT-spline discretization perfectly preserves its optimal high-order accuracy under aggressive local refinement; second, the presence of hanging nodes does not compromise the convergence behavior of the scheme.

\begin{figure}[htbp]
    \centering
    \includegraphics[width=0.32\textwidth]{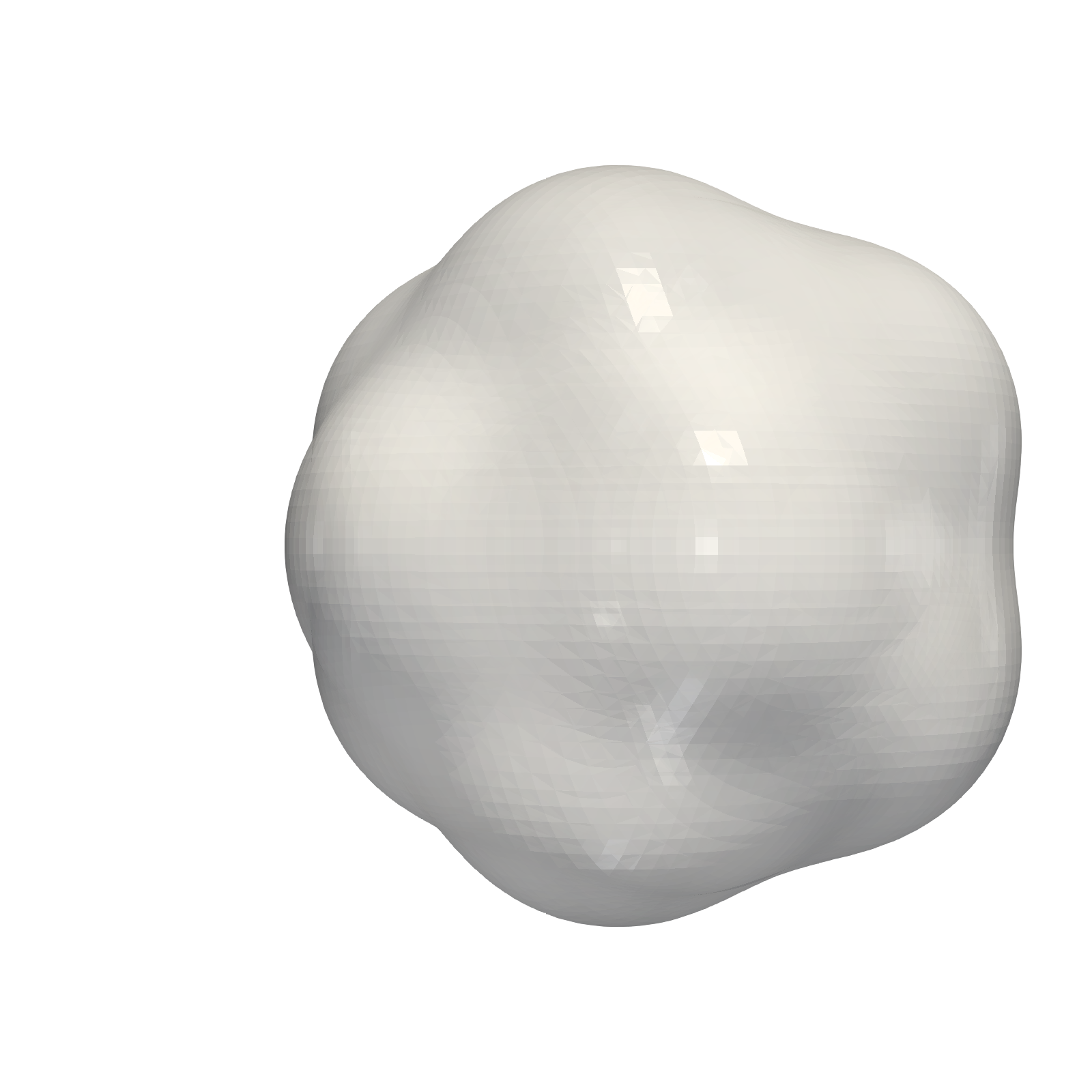}
    \hfill
    \includegraphics[width=0.32\textwidth]{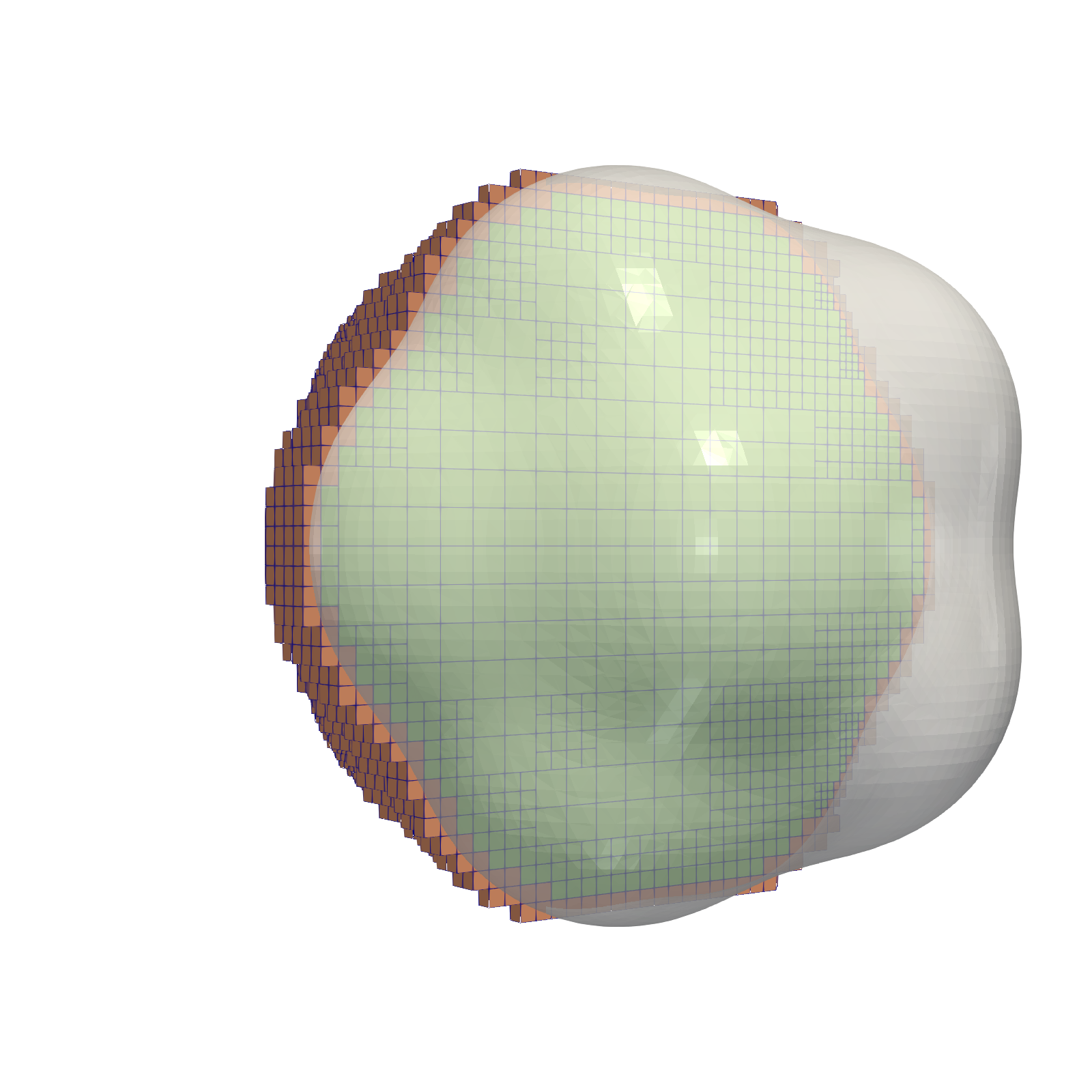}
    \hfill
    \includegraphics[width=0.32\textwidth]{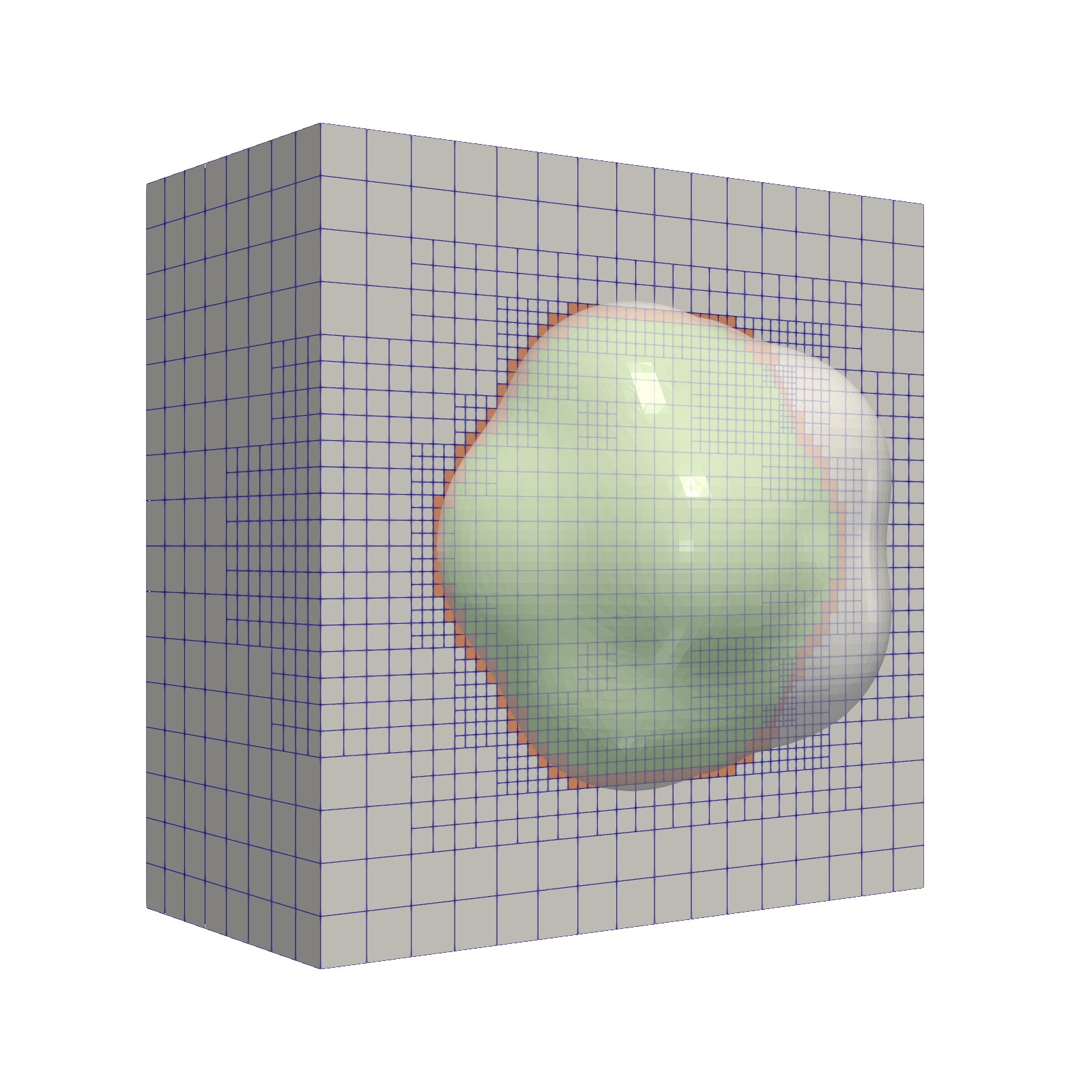}
    \caption{Adaptive mesh of Example~\ref{example3}. 
    Left: the surface $\Gamma = \{\phi = 0 \}$;
    Middle: half cut and interior cells;
    Right: half adaptive mesh with $\Gamma$. 
    Furthermore, gray, orange and green cells indicate the outer cells, the cut cells and the inner cells, respectively.}
    \label{Figure_Example3_Mesh}
\end{figure}

\begin{figure}[htbp]
    \centering
    \includegraphics[width=0.48\textwidth]{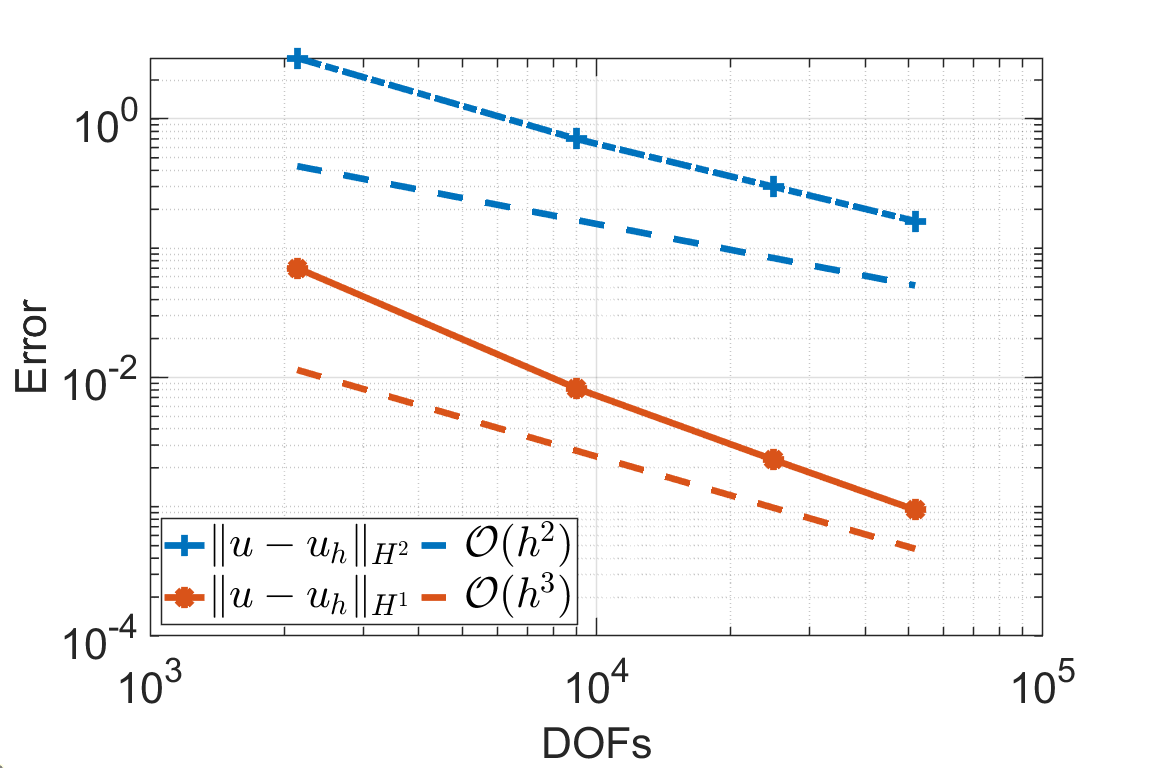}
    \hfill
    \includegraphics[width=0.48\textwidth]{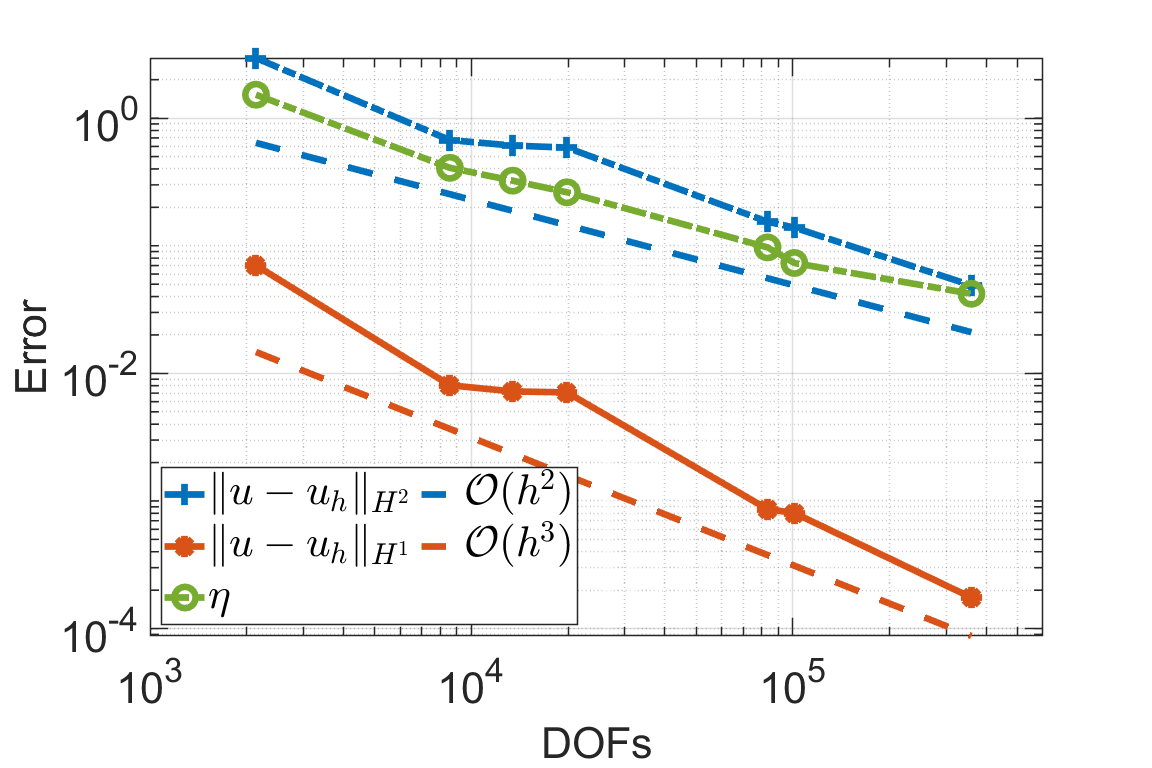}
    \caption{Convergence of errors with respect to DOFs on uniformly refined meshes (left) and adaptively refined meshes (right).}
    \label{Figure_Example3}
\end{figure}

\section{Conclusion}
\label{sec:conclusion}
We proposed and analyzed a least-squares UnCut FEM for the Poisson problem on level-set domains.
The method combines a $\phi$-FEM unfitted construction with $C^1$ continuous PHT-splines, which enables a fully primal formulation based on the $L^2$ residual $\|\Delta u_h+f\|_{L^2(\Omega_h)}$.
All terms are assembled on the active background domain $\Omega_h$ using standard element-wise quadrature, thereby avoiding geometric integration over curved boundaries or cut subregions of boundary elements. 
Consequently, the method is simple to implement and computationally efficient.
We proved existence and uniqueness of the discrete solution and established optimal-order \emph{a priori} error estimates in the $H^1$ norm.
In addition, a residual-based \emph{a posteriori} indicator was employed to guide local refinement on hierarchical T-meshes, and the numerical experiments demonstrate that it provides an effective criterion for adaptive mesh refinement.

In the future, by leveraging the exact geometric representation capabilities of PHT-splines, we plan to extend this unfitted least-squares framework 
to solve PDEs posed on surfaces and on trimmed CAD geometries.

\bibliographystyle{abbrv}
\bibliography{reference}

\end{document}

%% file: macros.tex
\providecommand*{\Dist}[2]{\operatorname{dist}({#1};{#2})}   % distance
\providecommand*{\Dist}[2]{\Dist{#1}{#2}}
\newcommand*{\TN}[1]{\big|\hspace{-0.3mm}\big\|{#1}\big\|\hspace{-0.3mm}\big|}     % Double bar norm
\newcommand*{\jump}[2][]{\left \llbracket{#2}\right\rrbracket_{#1}}

\newcommand{\be}{\begin{eqnarray}}
\newcommand{\ee}{\end{eqnarray}}

\newcommand{\ben}{\begin{eqnarray*}}
\newcommand{\een}{\end{eqnarray*}}